\documentclass[11pt]{article} 
\usepackage{multicol} 
\usepackage{newlfont}
\usepackage{amsmath,amssymb}
\usepackage{tikz}

\def\rit{\mathbb{R}}
\def\zit{\mathbb{Z}}   
\def\nit{\mathbb{N}}

\def\ppit{\mathbb{P}} 
\def\qit{\mathbb{Q}} 
\def\cit{\mathbb{C}}

\newcommand{\pf}{{\em Proof.~}}
\newcommand{\qed}{\hfill~~\mbox{$\Box$}}

\newenvironment{proof}{\smallskip \noindent \pf}{\qed \bigskip}

\newtheorem{theorem}{Theorem}[section]
\newtheorem{proposition}[theorem]{Proposition}
\newtheorem{definition}[theorem]{Definition}
\newtheorem{lemma}[theorem]{Lemma}
\newtheorem{corollary}[theorem]{Corollary}

\newtheorem{remark}[theorem]{Remark}
\newtheorem{example}[theorem]{Example}
\newtheorem{conjecture}[theorem]{Conjecture}

\DeclareMathOperator{\card}{Card}

\DeclareMathOperator{\vol}{vol}

\DeclareMathOperator{\gr}{gr}
\DeclareMathOperator{\Spec}{Spec}

\DeclareMathOperator{\Ehr}{Ehr}

\DeclareMathOperator{\Card}{Card}

\DeclareMathOperator{\re}{Re}

\begin{document}


\title{\bf From Hodge theory for tame functions to Ehrhart theory for polytopes}

\author{\sc Antoine Douai\thanks{Mathematics Subject Classification 52B20, 32S40, 14J33.
}\\ 
Universit\'e C\^ote d'Azur, CNRS, LJAD, FRANCE\\
Email address: antoine.douai@univ-cotedazur.fr}

\maketitle

\begin{abstract}
We study the interplay between Sabbah's mixed Hodge structure for regular functions and Ehrhart theory for polytopes. To this end, we analyze the properties of the Poincar\'e polynomial of the Hodge filtration of this mixed Hodge structure. We deduce various combinatorial properties of the Hodge numbers attached to a convenient and nondegenerate Laurent polynomial. 
\end{abstract}

\section{Introduction}

The aim of these notes is to study the interplay between the Hodge theory for regular functions developed by Sabbah in \cite{Sab0}, \cite{Sab} (the local model is explained in \cite{ScSt}) and the Ehrhart theory for polytopes. The link between Hodge theory and Ehrhart theory is certainly not new and has been considered in other contexts (see for instance \cite{DK}), but it is fruitful to emphasize it using the mixed Hodge structure basically inspired by 
the mixed Hodge structure on the cohomology of the Milnor fiber of an hypersurface singularity introduced by Steenbrink in \cite{Steen}, and more precisely by its description {\em via}
Varchenko's asymptotic Hodge structure \cite {Va}.

On the one hand, given a tame regular function $f$ on a smooth affine complex variety of dimension $n$, it is constructed in \cite{Sab0}, \cite{Sab}, using Fourier transform techniques, a (limit) mixed Hodge structure 
$MHS_f =(H, F^{\bullet}, W_{\bullet})$ where the (decreasing) Hodge filtration $F^{\bullet}$ and the (increasing) weight filtration $W_{\bullet}$ are filtrations on a finite dimensional complex vector space $H$ satisfying some compatibility conditions: this is {\em Sabbah's mixed Hodge structure}.
The Hodge filtration has a very concrete description, based on a global version of the Brieskorn lattice, and the weight filtration is a monodromy weight filtration (see Section \ref{sec:HodgeRegular} for details).
When $f$ is the mirror partner of a weighted projective space in the sense of \cite{DoMa}, explicit computations of these filtrations are carried out in \cite{DoSa2}. 

We first study the Poincar\'e polynomial of the Hodge filtration. Let us put $\theta_p :=\dim \gr_F^{n-p} H$ for $p\in\zit$. Then $\theta_p =0$ if $p>n$ or $p<0$ and we call the polynomial
\begin{equation}\nonumber
\theta (z) :=\sum_{p=0}^n \theta_i z^i 
\end{equation}
the $\theta$-{\em vector} of the function $f$. 
Thanks to the definition of the Hodge filtration, the $\theta$-vectors can be computed from the spectrum at infinity of $f$ (the definition is recalled in Section \ref{sec:HodgeRegular}) and this allows us to handle significant examples for which this spectrum is already known, in particular convenient and nondegenerate (Laurent) polynomials in the sense of Kouchnirenko \cite{K}, see \cite{DoMa}, \cite{DoSa1}, \cite{DoSa2} (in this situation, the spectrum at infinity is equal to the spectrum of a Newton filtration defined on the Milnor ring of $f$). 
Since $\theta (1)=\dim H\neq 0$, we also have
\begin{equation}\nonumber
\frac{\theta (z)}{(1-z)^{n+1}}=\sum_{m\geq 0} L_{\psi} (m) z^{m} 
\end{equation}
where $L_{\psi}$ is s polynomial in $m$ of degree $n$: we call it the {\em Hodge-Ehrhart polynomial of $f$}. 

In the first part of these notes we highlight the following properties of the $\theta$-vector and of the Hodge-Ehrhart polynomial of a function $f$:
\begin{itemize}
\item we have a decomposition $H=\oplus_{[0,1[} H_{\alpha}$ where $H_{\alpha}$ denotes a generalized eigenspace of the monodromy of $f$ at infinity and we can define similarly the $\theta$-vectors $\theta^{\alpha}(z)$ and the Hodge-Ehrhart polynomials $L_{\psi}^{\alpha}$ for $\alpha\in [0,1[$ by setting $\theta_p^{\alpha}:=\dim \gr_F^{n-p} H_{\alpha}$: 
if we write $\theta^{\neq 0} (z):=\sum_{\alpha \in ]0,1[} \theta^{\alpha} (z)$,
we get
\begin{equation}\nonumber
\theta (z)=\theta^0 (z) +\theta^{\neq 0} (z)
\end{equation}
where $z^n \theta (z^{-1}) =\theta (z)$ and $z^{n+1} \theta^{\neq 0} (z^{-1})=\theta (z)$, the two last equalities being a consequence of the symmetries of the Hodge numbers of the mixed Hodge structure $MHS_f$ (see
Proposition \ref{prop:DecompTheta}). If $H=H_0$ ({\em resp.} $H_0=0$) we will say that $f$ is {\em reflexive} ({\em resp.} {\em anti-reflexive}).
\item  Using again the symmetries of the Hodge numbers of the mixed Hodge structure $MHS_f$, we show the reciprocity law
\begin{equation}\label{eq:EhrReciprocityIntro}
L_{\psi}(-m)=(-1)^n  L_{\psi} (m) +(-1)^n (L_{\psi}^0 (m-1)-L_{\psi}^0 (m))
\end{equation}
for $m\geq 1$ where $L_{\psi}^0$ denotes a contribution of the unipotent part of the monodromy at infinity of $f$ (see Theorem \ref{theo:EhrReciprocity}).
\item We give informations about the coefficients and the roots of 
$L_{\psi}$, regarded as a polynomial over $\cit$. 
In particular, if $f$ is reflexive ({\em resp.} anti-reflexive) we get from (\ref{eq:EhrReciprocityIntro}) the equality 
$L_{\psi}(-m)=(-1)^n L_{\psi} (m-1)$ ({\em resp.} $L_{\psi}(-m)=(-1)^n  L_{\psi} (m)$): the roots of the Hodge-Ehrhart polynomial of a reflexive ({\em resp.} anti-reflexive) function are symmetrically distributed around the 
"critical line" $\re z =-1/2$ ({\em resp.}  $\re z=0$). In Section \ref{sec:CoeffRoots}, 
 we use Rodriguez-Villegas' trick \cite{RV} in order to test if the non integral roots of the Hodge-Ehrhart polynomials of such functions are indeed on these critical lines (this is the content of Theorem \ref{theo:Kron}; note that Hodge-Ehrhart polynomials may have integral roots, see Example \ref{ex:PolynomeHE}).
\item Last, we study the behaviour of the $\theta$-vectors under Thom-Sebastiani sums (Sabbah's mixed Hodge structures behave very well with respect to Thom-Sebastiani type problems): we show that 
if $f'$ is reflexive the  $\theta$-vector $\theta (z)$ of the Thom-Sebastiani sum $f'\oplus f''$ satisfies
\begin{equation}\label{eq:TSIntro}
\theta (z)= \theta' (z)\theta'' (z)
\end{equation}
where $\theta '(z)$ ({\em resp.} $\theta '' (z)$) denotes the $\theta$-vector of $f'$ ({\em resp.} $f''$), see Theorem \ref{theo:ThomSeb}. 
 This is useful in order to construct reflexive and anti-reflexive functions and functions whose Hodge-Ehrhart polynomials have their roots on the critical lines alluded to above. 
\end{itemize}

On the other hand, a function $f$
meets the classical Ehrhart theory {\em via} its Newton polytope. More precisely,
let $f$ be a convenient and nondegenerate Laurent polynomial (in the sense of Kouchnirenko \cite{K}) defined on $(\cit^*)^n$ and assume that its Newton polytope $P$ is simplicial.
 We readily get from \cite{D12} the equality 
\begin{equation}\label{eq:Mirror}
\theta (z)=\delta_{P}(z)
\end{equation}
where $\theta (z)$ is the $\theta$-vector of $f$ and $\delta_{P}(z)$ is the $\delta$-vector of the polytope $P$ (see Section \ref{sec:HodgeEhrhart}; the setting is different, but since we deal with Hodge numbers formula (\ref{eq:Mirror}) should be compared 
with the one obtained in \cite[Section 4]{DK}
for the Hodge numbers of a hypersurface in a torus). 
It is worth noticing here that the $\delta$-vectors are finally closely linked with the spectrum at infinity of $f$, in other words with Bernstein relations in the Fourier transform of the Gauss-Manin system of $f$. It also follows from (\ref{eq:Mirror}) that
the Hodge-Ehrhart polynomial  of the Laurent polynomial $f$ is equal to the Ehrhart polynomial of its Newton polytope $P$.
In this framework, the polynomials $L_{\psi}^{\alpha}$ correspond to the weighted Ehrhart polynomials
defined by Stapledon in \cite{Stapledon} and the reciprocity law (\ref{eq:EhrReciprocityIntro}) corresponds {\em via} equality (\ref{eq:Mirror}) to the classical Ehrhart reciprocity. Note that if (\ref{eq:EhrReciprocityIntro}) is basically Ehrhart reciprocity when $f$ is a Laurent polynomial, its combinatorial meaning  is less clear for other classes of tame functions (for instance polynomial functions).

In this way, we get from Hodge theory and singularity theory other interpretations of some very well-known results in combinatorics: among other examples (discussed in Section \ref{subsec:EhrRecipFunctionPolytope}), the formula given by Braun in \cite{Braun}  for the Ehrhart series of the free sum of two reflexive polytopes can be seen as a particular case of the Thom-Sebastiani formula (\ref{eq:TSIntro}).
Conversely, equality (\ref{eq:Mirror}) has also several consequences for Sabbah's mixed Hodge structures.  
For instance, using Hibi's Lower Bound Theorem \cite{Hibi1}, we get the
following Lower Bound Theorem for the Hodge numbers: if $f$ 
is a convenient and nondegenerate Laurent polynomial defined on $(\cit^* )^n$ 
we have
\begin{equation} \label{eq:LBT}
\dim\gr_F^{n-1}H\leq \dim\gr_F^{n-i}H
\end{equation} 
for $1\leq i\leq n-1$ (see Proposition \ref{prop:LowerBoundTheorem}). Inequality (\ref{eq:LBT}) is helpful in order to understand better the distribution of the spectrum at infinity of a tame regular function $f$: it implies that the number of spectral values contained in $]p-1 ,p]$ is greater or equal than the number of spectral values contained in $]0,1]$ if $p$ is an integer such that $1\leq p\leq n-1$. 

We expect that (\ref{eq:LBT}) is true for all tame functions: this is a part of what we call the $H$-conjecture (see Conjecture \ref{conj:H}). Note that the $\delta$-vector of a simplicial lattice polytope $P$ is a $\theta$-vector if and only if $P$ contains the origin as an interior point (see Proposition \ref{prop:DeltaThetaIFF}): since
Hibi's inequalites $\delta_1 \leq\delta_i$ for $i=1,\ldots n-1$  for the coefficients of the $\delta$-vector hold only for polytopes with interior lattice points, one may suspect that these inequalities have a Hodge theoretic flavor.
Of course, other results of this kind are possible since, by (\ref{eq:Mirror}), any (in)equality involving the $\delta$-vector of a polytope containing the origin as an interior point (this is the convenience assumption) has a counterpart for the Hodge numbers of a Sabbah's mixed Hodge structure. This is discussed in Section
\ref{sec:LinearIn}.

This paper is organized as follows: in Section \ref{sec:HodgeRegular},
we gather the results about Sabbah's mixed Hodge structures that we will use. 
We introduce a Hodge-Ehrhart theory for regular tame functions in Section \ref{sec:DeltaSHM}. We study the coefficients and the roots of Hodge-Ehrhart polynomials in Section \ref{sec:CoeffRoots}. 
Section \ref{sec:ThomSeb} is devoted to the study of Thom-Sebastiani sums
 and we explain in Section \ref{sec:HodgeEhrhart} how the previous results are related to the classical Ehrhart theory for polytopes. In the last section, we study some linear inequalities among the Hodge numbers of a tame function.

\section{Framework: Sabbah's mixed Hodge structures for regular functions}
\label{sec:HodgeRegular}

In this section, we recall Sabbah's setting \cite{Sab0}, \cite{Sab}. 
Let $f$ be a regular function on a smooth affine complex variety $U$ of dimension $n\geq 2$. We will denote by $G$ the localized Laplace transform of its Gauss-Manin system and by $G_0$ its Brieskorn module. By the very definition, 
\begin{equation}\nonumber
G= \Omega^n (U) [\theta , \theta^{-1}]/(\theta d-df\wedge ) \Omega^{n-1} (U) [\theta , \theta^{-1}]
\end{equation}
and $G_0$ is the image of  $\Omega^n (U) [\theta ]$ in $G$.
If $f$ has only isolated singularities, we have
\begin{equation}\nonumber
G_0 = \Omega^n (U) [\theta ]/(\theta d-df\wedge ) \Omega^{n-1} (U) [\theta ]
\end{equation}
and 
\begin{equation}\nonumber
G_0 /\theta G_0 \cong \Omega^{n}(U)/df\wedge \Omega^{n-1} (U),
\end{equation}
the latter being a finite dimensional vector space of dimension $\mu$, the global Milnor number of $f$.
We will say that $f$ is {\em tame} if $f$ is cohomologically tame in the sense of  \cite[Section 8]{Sab}, that is if for some compactification of $f$ no modification of the cohomology of the fibers comes from infinity. 
If $f$ is tame, $G_0$ is a {\em lattice} in $G$ {\em i.e} $G_0$ is free over $\cit [\theta]$ (see \cite[Corollary 10.2]{Sab}) and we have $\cit [\theta , \theta^{-1}]\otimes_{\cit [\theta ]} G_0 =G$. 
 Our favourite class of tame functions is provided by Kouchnirenko's convenient and nondegenerate (Laurent) polynomials defined in \cite{K}: in this situation,
the freeness of $G_0$ is shown by elementary means in \cite[Section 4]{DoSa1}.  

{\em From now on, and otherwise stated, $f$ denotes a tame regular function on a smooth affine complex variety of dimension $n$}.

Let $\tau :=\theta^{-1}$. Then $G$ is a free $\cit [\tau ,\tau^{-1}]$-module, equipped with a derivation $\partial_{\tau}$ (induced from the formula $\partial_{\tau} [\omega ]=[-f\omega]$ if $\omega\in\Omega^{n} (U)$ where the brackets denote the class in $G$).
Let $V_{\bullet}G$ be the Kashiwara-Malgrange filtration of $G$ along $\tau =0$ as in \cite[Section 1]{Sab}. This is an increasing filtration, indexed by $\qit$, which satisfies the following properties:
\begin{itemize}
\item for every $\alpha\in\qit$, $V_{\alpha} G$ is a free $\cit [\tau]$-module of rank $\mu$,
\item $\tau V_{\alpha} G \subset V_{\alpha -1} G$ and $\partial_{\tau} V_{\alpha} G \subset V_{\alpha +1} G$,
\item $\tau\partial_{\tau} +\alpha$ is nilpotent on $\gr_{\alpha}^V G$.
\end{itemize} 
If $f$ is a convenient and nondegenerate (Laurent) polynomial, the filtration $V_{\bullet}G$ can be computed with the help of a Newton filtration (suitably normalized), see \cite[Lemma 4.11]{DoSa1}.
We define, for $\alpha\in\qit$,
\begin{equation}\nonumber
H_{\alpha} := \gr_{\alpha}^V G\ \mbox{and}\ H:=\oplus_{\alpha\in [0,1[} H_{\alpha}.
\end{equation}
We will write $H = H_{\neq 0}\oplus H_0$,
where $H_{\neq 0}:= \oplus_{\alpha \in ]0,1[} H_{\alpha}$.
The $\cit$-vector space $H$ (of dimension $\mu $) is equipped with a nilpotent endomorphism $N$ induced by $-(\tau\partial_{\tau}+\alpha )$ on $H_{\alpha}$: we have $N^n =0$ on $H_{\alpha}$ if $\alpha\in ]0,1[$ and $N^{n+1} =0$ on $H_0$. The vector space $H$ can be identified with the relative cohomology 
$H^n (U , f^{-1} (t) ; \cit )$ for $|t|>>0$: in this setting, $H_{\alpha}$ corresponds to the generalized eigenspace of the monodromy at infinity associated with the eigenvalue $\exp (2i\pi\alpha )$ and the unipotent part of this monodromy is equal to $\exp (2i\pi N)$.

We first construct a Hodge filtration on $H$. Let $G_{\bullet}$ be the increasing filtration of $G$ defined by $G_p :=\theta^{-p} G_0 =\tau^p G_0$ and
\begin{equation}\nonumber
G_p H_{\alpha} := G_p \cap V_{\alpha} G/G_p \cap V_{<\alpha} G
\end{equation}
for $\alpha\in [0,1[$ and $p\in\zit$.

\begin{definition} The Hodge filtration on $H$ is the decreasing filtration $F^{\bullet}$, indexed by $\zit$, defined by 
$F^p H_{\alpha} := G_{n-1-p}H_{\alpha}$ if $\alpha\in ]0,1[$ and $F^p H_{0} := G_{n-p}H_{0}$. 
\end{definition}

 The next lemma makes the link between the Hodge filtration and  the spectrum at infinity $\Spec_f (z)$ of $f$, 
defined by
\begin{equation}\label{eq:spectrum}
\Spec_f (z) :=\sum_{\beta\in\qit} \dim_{\cit} \gr^V_{\beta} (G_0 /\theta G_0 ) z^{\beta}
\end{equation}
where $\gr^V_{\beta} (G_0 /\theta G_0 )= V_{\beta} G_0 /(V_{<\beta} G_0 +\theta G_0 \cap V_{\beta} G_0 )$
and $V_{\beta} G_0 = G_0 \cap V_{\beta} G$.

\begin{lemma}\label{lemma:IsoGrFV}
The multiplication by $\tau^{-1}$ induces isomorphisms
\begin{equation}\nonumber
\tau^{-(n-1-p)}:\gr_F^p H_{\alpha} \stackrel{\cong}{\longrightarrow} \gr^V_{\alpha+n-1-p} (G_0/ \theta G_0 )
\end{equation}
if $\alpha\in ]0,1[$ and 
\begin{equation}\nonumber
\tau^{-(n-p)}:\gr_F^p H_{0} \stackrel{\cong}{\longrightarrow} \gr^V_{n-p} (G_0/ \theta G_0 ). 
\end{equation}
\end{lemma}
\begin{proof} The multiplication by $\tau^{-p}$ induces an isomorphism 
\begin{equation}\nonumber
\tau^{-p}: V_{\alpha}G\cap G_p \stackrel{\cong}{\longrightarrow} V_{\alpha +p} G \cap G_0 
\end{equation}
hence an isomorphism from $\gr^G_p H_{\alpha}$ onto $\gr^V_{\alpha +p} (G_0/ \theta G_0 )$ for $\alpha\in [0,1[$ (see also \cite[Section 1]{Sab}).
\end{proof}

\begin{lemma}\label{lemma:Ffinite}
We have
\begin{enumerate}
\item $F^{n} H_{\alpha}=0$ for $\alpha\in ]0,1[$ and $F^{n+1} H_{0}=0$,
\item $F^0 H_{\alpha}=H_{\alpha}$ for $\alpha\in [0,1[$.
\end{enumerate}
\end{lemma}
\begin{proof} 
By \cite[13.18]{Sab}, we have $G_k H_{\alpha}=0$ for $k<0$ whence the first point. 
Since the spectrum at infinity is symmetric about $n/2$ (see {\em loc. cit.}), we deduce that  $\dim_{\cit} \gr^V_{\beta} (G_0 /\theta G_0 )=0$ if $\beta\notin [0,n]$ and we get the second point using Lemma \ref{lemma:IsoGrFV}.
\end{proof}

\begin{remark} \label{rem:theta0}
We have $\dim \gr_F^n H= \dim \gr_F^n H_0 =\dim \gr^V_0 (G_0 / \theta G_0 )$ and the following cases may occur:
\begin{itemize}
\item $\dim\gr_F^n H=1$, this is the case if $f$ is a convenient and nondegenerate Laurent polynomial defined on $(\cit^*)^n$ (see \cite[Lemma 4.3]{DoSa1}),
\item $\dim\gr_F^n H=0$, this is the case if $f$ is a tame polynomial defined on $\cit^n$ (see \cite[Corollary 13.2]{Sab}),
\item $\dim\gr_F^n H>1$, this is the case if $f:U\rightarrow\cit$ is defined by $f(u_0 ,\ldots, u_n )=
u_0 +\ldots +u_n$ where $U=\{ (u_0 ,\ldots ,u_n )\in \cit^{n+1}, \ u_0^{w_0}\ldots u_n^{w_n}=1\}$ and 
$(w_0 ,\ldots ,w_n )\in\nit^{n+1}$ is such that $\gcd (w_0 ,\ldots ,w_n )>1$, see \cite[Section 6.1]{Mann}.
\end{itemize} 
\end{remark}

Let us come to the weight filtration. We have on $H$ a real structure coming from the identification 
\begin{equation}\nonumber
H\stackrel{\cong}{\longrightarrow} H^n (U, f^{-1}(t) ;\cit )=\cit\otimes_{\rit} H^n (U, f^{-1}(t) ;\rit ).
\end{equation}
We will write $H=\cit\otimes H^{\rit}$ and, if $E$ is a subspace of $H$, we will denote by $\overline{E}$ the conjugate of $E$ defined by this complexification.
 Due to the fact that the monodromy at infinity is defined over $\rit$,
we have $\overline{H_{\alpha}}=H_{1-\alpha}$ if $\alpha\in ]0,1[$ and $\overline{H_{0}}=H_{0}$
and therefore the decomposition $H^{\rit}=H_0^{\rit}\oplus H_{\neq 0}^{\rit}$.  
Recall that if $L$ is a nilpotent endomorphism of a vector space $G$ with $L^{r+1}=0$, the {\em weight filtration of $L$ centered at $r$} is the unique increasing filtration $W_{\bullet}$ of $G$ 
\begin{equation}\nonumber
0\subset W_0 \subset W_1 \subset \ldots \subset W_{2r} =G
\end{equation}
such that 
\begin{equation}\label{eq:CarWeightN1}
L(W_i )\subset W_{i-2}
\end{equation}
and
\begin{equation}\label{eq:CarWeightN2}
L^{\ell} :  W_{r+\ell}/ W_{r+\ell -1} \stackrel{\cong}{\longrightarrow}  W_{r-\ell}/ W_{r-\ell -1}.
\end{equation}
We apply this construction to the nilpotent endomorphism $2i\pi N$ of $H^{\rit}$:

\begin{definition}
The weight filtration $W_{\bullet}$ on $H^{\rit}$ is the increasing weight filtration of $2i\pi N$ centered at $n-1$ on $H_{\neq 0}^{\rit}$ and centered at $n$ on $H_0^{\rit}$.
\end{definition}

The Hodge filtration, the weight filtration (transferred on $H$; we will denote by the same letter the weight filtration and its complexification)
and the conjugation are related:
 
\begin{theorem}\cite[Theorem 13.1]{Sab}\label{theo:CanonicalMHSf}
The triple $MHS_f :=(H, F^{\bullet}, W_{\bullet} )$ is a mixed Hodge structure: for $k\in\zit$, the induced filtration $F^{\bullet} \gr^{W}_k H$ defines a Hodge structure of weight $k$ on $\gr^W_k H$.
\end{theorem}

\noindent We will call $MHS_f$ the {\em Sabbah's mixed Hodge structure of} $f$. See \cite{DoSa2} for a concrete description of this mixed Hodge structure in a particular case.

\begin{remark} \label{rem:Ipq}
Following \cite{Del}, let us consider the subspaces
\begin{equation}\label{eq:Ipq}
I^{p,q} := (F^p\cap W_{p+q})\cap (\overline{F^q}\cap W_{p+q} + \sum_{j>0}\overline{F^{q-j}}\cap W_{p+q-j-1} ).
\end{equation} 
By loc. cit.,
we have the decompositions
\begin{equation}\label{eq:DecIpq3}
W_m H=\oplus_{i+j\leq m} I^{i,j},\
F^p H=\oplus_{i\geq p, j} I^{i,j}\ \mbox{and}\
H=\oplus_{p,q} I^{p,q}.
\end{equation}
Since $N (F^{p})\subset F^{p-1}$ and $N (W_q )\subset W_{q-2}$, we get
$N (I^{p,q})\subset I^{p-1,q-1}$.
It follows that $N$ is strict with respect to $F^{\bullet}$ and induces on $\gr_F H$ a graded morphism $[N]$ of degree $-1$: we will use these facts below.
\end{remark}

The {\em  Hodge numbers} of the mixed Hodge structure $MHS_f =(H,F^{\bullet} , W_{\bullet})$
are defined by
\begin{equation}\label{eq:Defhpq}
h^{p,q}:=\dim \gr_F^p \gr^{W}_{p+q}H.
\end{equation}
We will also refer to it as the {\em Hodge numbers of the function} $f$ (if necessary, we will write $h^{p,q}(f)$ instead of $h^{p,q}$). 
Note that we have $\dim \gr_F^p H=\sum_{q} h^{p,q}$ and $\dim \gr^W_m H=\sum_{q} h^{q,m-q}$.
 We will write $h^{p,q}=h^{p,q}_{0}+h^{p, q}_{\neq 0}$, where the subscript refers to the decomposition $H =H_0\oplus H_{\neq 0}$.

\begin{proposition}\label{prop:SymHodgeNumbers}
We have
\begin{enumerate}
\item $h^{p,q}=h^{q,p}$,
\item $h^{p,q}_{\neq 0}=h^{n-1-q, n-1-p}_{\neq 0}$ and $h^{p,q}_{0}=h^{n-q, n-p}_{0}$.
\item $h^{p,q}=0$ if $p,q \notin [0,n]$.
\end{enumerate}
\end{proposition}
\begin{proof} 
By Theorem \ref{theo:CanonicalMHSf}, we have
\begin{equation}\nonumber
\dim \gr_F^p \gr^{W}_{p+q}H=\dim F^p \gr^{W}_{p+q}\cap \overline{F^{q}\gr^W_{p+q}}H=
\dim \overline{F^p \gr^{W}_{p+q}}\cap F^{q}\gr^W_{p+q}H= \dim \gr_F^q \gr^{W}_{p+q}H
\end{equation}
and this shows the first point. 
For the second one, we use the isomorphisms
\begin{equation}\label{eq:IsoN}
[N]^{p+q-(n-1)} : \gr_F^p  \gr^W_{p+q} H_{\neq 0} \stackrel{\cong}{\longrightarrow}  \gr_F^{n-1-q}  \gr^W_{2(n-1)-(p+q)}H_{\neq 0}
\end{equation}
 if $p+q\geq n-1$ which follow from the fact that $N$ is strict with respect to $F^{\bullet}$ (see Remark \ref{rem:Ipq}) and the characteristic property (\ref{eq:CarWeightN2}) of the weight filtration. Same thing for $h^{p,q}_{0}$, keeping in mind that $W_{\bullet}$ is centered at $n$ in this case. The last assertion follows from
Lemma \ref{lemma:Ffinite} and the previous symmetry properties.
\end{proof}

According to the decomposition $H=\oplus_{\alpha \in [0,1[} H_{\alpha}$, 
we define
$$H^{p,q}_{\alpha}:=\gr_F^p \gr^W_{p+q} H_{\alpha},\ 
H^{p,q}:=\oplus_{\alpha\in [0,1[}H^{p,q}_{\alpha}\ \mbox{and}\ h^{p,q}_{\alpha}:=\dim H^{p,q}_{\alpha}$$
(the filtrations are defined from the beginning on each summand).
We will use the following refinement of Proposition \ref{prop:SymHodgeNumbers}:

\begin{lemma}\label{lemma:SymhAlpha}
We have
$h^{p,q}_{\alpha}=h^{n-1-p, n-1-q}_{1-\alpha}$ if $\alpha \in ]0,1[$ and  $h^{p,q}_{0}=h^{n-p, n-q}_{0}$.
\end{lemma}
\begin{proof}
On the one hand, the isomorphisms
\begin{equation}\nonumber
[N]^{p+q-(n-1)} : \gr_F^p  \gr^W_{p+q} H_{\alpha} \stackrel{\cong}{\longrightarrow}  \gr_F^{n-1-q}  \gr^W_{2(n-1)-(p+q)}H_{\alpha}\ \mbox{for}\ \alpha \in ]0,1[\ \mbox{and}\ p+q\geq n-1
\end{equation}
and 
\begin{equation}\nonumber
[N]^{p+q-n} : \gr_F^p  \gr^W_{p+q} H_{0} \stackrel{\cong}{\longrightarrow}  \gr_F^{n-q}  \gr^W_{2n-(p+q)}H_{0}\ \mbox{for}\ p+q\geq n 
\end{equation}
give
$h^{p,q}_{\alpha}=h^{n-1-q, n-1-p}_{\alpha}$ if $\alpha \in ]0,1[$
and 
$h^{p,q}_{0}=h^{n-q, n-p}_{0}$. On the other hand, $\overline{H^{p,q}_{\alpha}}=H^{q,p}_{1-\alpha}$ if $\alpha\in ]0,1[$ and
$\overline{H^{p,q}_{0}}=H^{q,p}_{0}$ because $\overline{H_{\alpha}}=H_{1-\alpha}$ if $\alpha\in ]0,1[$ and $\overline{H_0}=H_0$.
Thus
$h^{p,q}_{\alpha}=h^{q,p}_{1-\alpha}$ if $\alpha \in ]0,1[$
and
$h^{p,q}_{0}=h^{q,p}_{0}$. The result follows.
\end{proof}

\begin{example}
\label{ex:BasicLaurent}
Let $f$ be the Laurent polynomial defined on $(\cit^*)^5$ by
$$f(u_1 ,\ldots ,u_5 )=u_1+u_2 +u_3 +u_4 +u_5 +\frac{1}{u_1^3 u_2^3 u_3^3 u_4^3 u_5^4}.$$
Then $\mu =17$ and the Hodge diamond of the mixed Hodge structure $MHS_f$ is

$$\begin{array}{ccccccccccc}
 & &   &  & & 1 & &  & &   &\\
 & &   &  & 1 &  & 1  &  &  &  & \\
 & &  & 0 &  & 1 &  & 0 & &  &\\
 &  & 0 &  & 1 &  & 1  &  & 0 &  & \\
 & 0 &  & 1 &  & 2& & 1&  & 0  & \\
0&  & 0  &  & 1  & & 1 & & 0 &   & 0 \\
& 0 &   & 0 & & 1 & & 0 & & 0  &\\
 & &  0 &  & 1 &  & 1  &  & 0 &  & \\
 & &  & 0 &  & 1 &  & 0 & &  &\\
 &  &  &  & 0 &  & 0  &  &  &  & \\
 &  &  &  &  & 1 & & &  &   & \\
\end{array}$$

\noindent where $h^{0,0}$ is at the top (we use the description of Sabbah's mixed Hodge structure given in \cite{DoSa2} in order to compute the Hodge numbers). Note that the function $f$ is the mirror partner of the weighted projective space $\ppit (1, 3,3,3,3, 4)$ in the sense of \cite{DoMa}.

\end{example}

\begin{remark}\label{rem:HodgeTate}
A natural question is to ask when the Hodge numbers satisfy $h^{p,q}=0$ for $p\neq q$. This is a characteristic property of mixed Hodge structure of Hodge Tate type: see \cite{D13}, \cite{Sab1} for a discussion in our singularity context. 
\end{remark}

\section{Hodge-Ehrhart theory for tame regular functions} 
\label{sec:DeltaSHM}

In this section, $f$ denotes a tame regular function on a smooth affine complex variety of dimension $n$ and $MHS_f =(H, F^{\bullet}, W_{\bullet})$ denotes its Sabbah's mixed Hodge structure. 

\subsection{$\theta$-vectors}
\label{sec:EhrSHM}

Let us define, for $p\in\zit$,
$\theta_p^{\alpha}:= \dim \gr_F^{n-p} H_{\alpha}$ if $\alpha\in [0,1[$ 
and $\theta_p :=\sum_{\alpha\in [0,1[}\theta_p^{\alpha}= \dim \gr_F^{n-p} H$. 
It follows from Lemma \ref{lemma:Ffinite} that $\theta_p=0$ for $p\notin [0,n]$.

\begin{definition} The $\theta$-vector of $f$ is the polynomial
$\theta (z) :=\sum_{p=0}^n \theta_p  z^{p}$.
\end{definition}

\noindent For $\alpha\in [0,1[$, we will also consider the polynomials
$\theta^{\alpha} (z) :=\sum_{p=0}^n \theta_p^{\alpha} z^{p}$: by the very definition, we have $\theta (z)=\sum_{\alpha\in [0,1[} \theta^{\alpha} (z)$.

\begin{lemma}\label{lemma:SymDelta}
We have the following symmetry properties:
\begin{enumerate}
\item $\theta_p^{\alpha} = \theta_{n-p+1}^{1-\alpha}$ 
and $\theta^{\alpha} (z)= z^{n+1}\theta^{1-\alpha}(z^{-1})$ if $\alpha\in ]0,1[$,
\item $\theta_p^{0} = \theta_{n-p}^{0}$ and  
$\theta^{0} (z) = z^n \theta^{0}(z^{-1})$.
\end{enumerate}
\end{lemma}
\begin{proof} Let $\alpha\in ]0,1[$.
By Proposition \ref{prop:SymHodgeNumbers} and Lemma \ref{lemma:SymhAlpha}, 
\begin{equation}\nonumber
\theta^{\alpha}_p =\dim\gr_F^{n-p} H_{\alpha}=\sum_q h^{n-p,q}_{\alpha} =\sum_q h^{p-1, n-1-q}_{1-\alpha}= \dim\gr_F^{p-1} H_{1-\alpha}=\theta^{1-\alpha}_{n-p+1}.
\end{equation}
Therefore,
$$\theta^{\alpha} (z)= \sum_{p=1}^n \theta_p^{\alpha} z^{p}=
\sum_{p=1}^n \theta_{n-p+1}^{1-\alpha} z^{p}=\sum_{p=1}^n \theta_{p}^{1-\alpha} z^{n-p+1}
=z^{n+1}\theta^{1-\alpha}(z^{-1}).$$
Analogous computations for $\theta_p^{0}$ and $\theta^{0} (z)$.
\end{proof}

\begin{proposition}\label{prop:DecompTheta}
We have the unique decomposition into polynomials with nonnegative coefficients
\begin{equation}\label{eq:Decomposition} 
\theta (z)=\theta^0 (z)+\theta^{\neq 0} (z)
\end{equation}
where $\theta^{0} (z)= z^{n}\theta^{0}(z^{-1})$
and $\theta^{\neq 0} (z)= z^{n+1}\theta^{\neq 0}(z^{-1})$. 
\end{proposition}
\begin{proof} Let us put $\theta^{\neq 0} (z):=\sum_{\alpha\in ]0,1[} \theta^{\alpha} (z)$: the decomposition (\ref{eq:Decomposition}) follows from the definitions and we get the symmetry properties from Lemma \ref{lemma:SymDelta}. The unicity of the decomposition follows from the required symmetry properties.
\end{proof}

\begin{corollary}\label{coro:SymDeltaH0}
We have $\theta_{n-p}=\theta_p$ for $p=0,\ldots ,n$ if and only if $H=H_0$.
\end{corollary}
\begin{proof}
We have $\theta_{n-p}=\theta_p$ for $p=0,\ldots ,n$ if and only if 
$z^{n}\theta (z^{-1})=\theta (z)$, that is if and only if $\theta^{\neq 0} (z)=0$ by the previous proposition. This happens if and only if $\theta^{\alpha} (z) =0$ for all $\alpha\in ]0,1[$, that is if and only if $H_{\alpha}=0$ for all $\alpha\in ]0,1[$.
\end{proof}

\begin{remark} \label{rem:SpectreSym}
We have $\theta_p^{\alpha}= \dim \gr^V_{\alpha +p-1} (G_0 /\theta G_0 )$ if $\alpha\in ]0,1[$ and $\theta_p^{0}= \dim \gr^V_{p} (G_0 /\theta G_0 )$ (see Lemma \ref{lemma:IsoGrFV}). Therefore, the $\theta$-vector can be computed from the spectrum at infinity of $f$ defined by (\ref{eq:spectrum}). 
\end{remark}

\begin{example}\label{ex:BasicLaurentBis}
Let $f$ be the Laurent polynomial of Example \ref{ex:BasicLaurent}. By \cite{DoSa2} we have 
$$\Spec_f (z)= 1+z+z^2 +z^3 +z^4 +z^5 +z^{7/4}+z^{4/3}+z^{7/3}+z^{10/3}+
z^{13/3}+z^{5/2}+z^{2/3}+z^{5/3}+z^{8/3}+z^{11/3}+z^{13/4}$$
therefore, by Remark \ref{rem:SpectreSym},
$$\theta^{0} (z)=1+z+z^2 +z^3 +z^4 +z^5 ,\ \theta^{\neq 0} (z)=z+3z^2 +3z^3 +3z^4 +z^5 $$ 
and
$$\theta (z)= 1+2z+4z^2+4z^3 +4z^4+2 z^5 .$$ 
See Section \ref{sec:Distrib} for a discussion about the unimodality of the $\theta$-vectors.
\end{example}

\subsection{Hodge-Ehrhart polynomials}
\label{sec:EhrPolSHM}

For $\alpha\in [0,1[$, let $L^{\alpha}_{\psi}$ be the function defined on the non negative integers by the formula
\begin{equation}\nonumber 
\sum_{m\geq 0} L^{\alpha}_{\psi} (m) z^{m} =\frac{\theta_0^{\alpha}+\theta_1^{\alpha} z+\ldots +\theta_n^{\alpha} z^n }{(1-z)^{n+1}}
\end{equation}
where the numbers $\theta_p^{\alpha}$ are defined as in Section \ref{sec:EhrSHM}.

\begin{lemma}\label{lemma:EhrphiPol}
Let $\alpha\in [0,1[$ such that $H_{\alpha}\neq 0$. Then, $L^{\alpha}_{\psi}(m)$ is a polynomial 
in $m$ of degree $n$.
\end{lemma}
\begin{proof}
Expanding the right hand term into a binomial series, we get
\begin{equation}\nonumber
L^{\alpha}_{\psi} (m) =\theta_0^{\alpha}\binom{m+n}{n}+\theta_1^{\alpha}\binom{m+n-1}{n}+\ldots +\theta_{n-1}^{\alpha}\binom{m+1}{n}+\theta_n^{\alpha}\binom{m}{n}.
\end{equation}
Therefore $L^{\alpha}_{\psi} (m)$ is a polynomial of degree $n$ in $m$ if $\dim H_{\alpha}= \theta_0^{\alpha}+\theta_1^{\alpha} +\ldots +\theta_n^{\alpha} \neq 0$.
\end{proof}

\begin{definition} \label{def:HEPol}
The
Hodge-Ehrhart polynomial of $f$
is the polynomial 
$L_{\psi}:=\sum_{\alpha \in [0,1[} L^{\alpha}_{\psi}$.
\end{definition}

It follows from Lemma \ref{lemma:EhrphiPol} that $L^{\alpha}_{\psi}$ and $L_{\psi}$
can be extended (as polynomials) to all 
$\cit$. In particular, it is possible to 
evaluate $L^{\alpha}_{\psi}(t)$ and $L_{\psi}(t)$ for negative integral values of $t$. 
We have the following "reciprocity laws" for the Hodge-Ehrhart polynomials:

\begin{theorem}\label{theo:EhrReciprocity} 
Let $f$ be a tame regular function on a smooth affine complex variety of dimension $n$. Then, for $m\geq 1$, 
\begin{enumerate}
\item $L^{\alpha}_{\psi} (-m)= (-1)^n L^{1-\alpha}_{\psi} (m)$
if $\alpha\in ]0,1[$,
\item $L^{0}_{\psi} (-m)= (-1)^n L^{0}_{\psi} (m-1)$,
\item $L_{\psi}(-m)=(-1)^n  L_{\psi} (m) +(-1)^n (L_{\psi}^0 (m-1)-L_{\psi}^0 (m))$.
\end{enumerate}
\end{theorem}
\begin{proof} 
Let us put $\Ehr_{\psi} (z) :=\sum_{m\geq 0} L_{\psi} (m) z^m$ 
and $\Ehr_{\psi}^{\alpha} (z):=\sum_{m\geq 0} L^{\alpha}_{\psi} (m) z^{m}$ for $\alpha\in [0,1[$.
Let $\alpha\in ]0,1[$: since $\theta_0^{\alpha}=0$ by Lemma \ref{lemma:SymDelta}, we have
\begin{equation}\nonumber
\Ehr_{\psi}^{\alpha} (z) =\frac{\theta_1^{\alpha} z+\ldots +\theta_n^{\alpha} z^n }{(1-z)^{n+1}}
\end{equation}
and we get
$$\Ehr_{\psi}^{\alpha}  (z^{-1})=(-1)^{n+1} \Ehr_{\psi}^{1-\alpha} (z)$$
using the symmetry property $\theta_p^{\alpha} = \theta_{n-p+1}^{1-\alpha}$ given by Lemma \ref{lemma:SymDelta}.
By Popoviciu's theorem (see \cite[Theorem 4.6]{Sta}), we also have 
$$\sum_{m\leq -1} L^{\alpha}_{\psi} (m) z^m = - \sum_{m\geq 0} L^{\alpha}_{\psi} (m) z^m =-\Ehr_{\psi}^{\alpha} (z),$$
 thus $\Ehr_{\psi}^{\alpha} (z^{-1})=-\sum_{m\geq 1} L^{\alpha}_{\psi} (-m) z^m$.
Putting these observations together, we get 
$$\sum_{m\geq 1} L^{\alpha}_{\psi} (-m) z^m = (-1)^n \sum_{m\geq 1} L^{1-\alpha}_{\psi} (m) z^m $$
and the first assertion follows. Analogous computations for $L^{0}_{\psi}$, taking now into account the symmetry property $\theta^0_p =\theta^0_{n-p}$ which shows that $\Ehr_{\psi}^{0}  (z^{-1})=(-1)^{n+1} z \Ehr_{\psi}^{0} (z)$. Last we get, using the two previous results, 
$$L_{\psi}(-m)=\sum_{0\leq \alpha <1} L_{\psi}^{\alpha} (-m)=
(-1)^n \sum_{0<\alpha <1}  L_{\psi}^{\alpha} (m) +(-1)^n L_{\psi}^0 (m-1)$$
$$=
(-1)^n \sum_{0\leq \alpha <1}  L_{\psi}^{\alpha} (m) +(-1)^n (L_{\psi}^0 (m-1)-L_{\psi}^0 (m))=
(-1)^n  L_{\psi} (m) +(-1)^n (L_{\psi}^0 (m-1)-L_{\psi}^0 (m))$$
for $m\geq 1$. This shows the last assertion.
\end{proof}

\begin{example}\label{ex:PolynomeHE}
Let $f$ be the polynomial on $\cit^2$ defined by
$f(u_1 ,u_2 )=u_1^2 +u_2^2 +u_1^2 u_2^2 $.
 Then $\mu =5$, $\Spec_{f} (z)=z^{1/2} +3z +z^{3/2}$ and
$H=H_0 \oplus H_{1/2}$.  We get, using Remark \ref{rem:SpectreSym},
$$\theta^0_0 =0,\  \theta^0_1 =3,\ \theta^0_2 =0,\ \theta^{1/2}_0 =0\ \mbox{and}\ \theta^{1/2}_1 =\theta^{1/2}_2 =1.$$
Therefore, 
$$L_{\psi}^{0} (m) = \frac{3}{2} (m^2 +m),\ L_{\psi}^{1/2} (m) =m^2\ \mbox{and}\ L_{\psi}(m) = \frac{5}{2} m^2 +\frac{3}{2}m.$$
Note that $L^{1/2}_{\psi} (-m)= L^{1/2}_{\psi} (m)$ and
$L^{0}_{\psi} (-m)=L^{0}_{\psi} (m-1)$, and this agrees with Theorem \ref{theo:EhrReciprocity}.
\end{example}

\subsection{Reflexive and anti-reflexive functions}

Theorem \ref{theo:EhrReciprocity} shows that the vector space $H_0$ plays a special role in the study of Hodge-Ehrhart polynomials. This motivates the following definition (see Remark \ref{rem:HegalH0} for a justification of the terminology):

\begin{definition}\label{def:RefARef}
We will say that $f$ is reflexive if $H= H_0$ and that $f$ is anti-reflexive if $H_0 =0$.
\end{definition}

\noindent Recall the spectrum at infinity $\Spec_f (z)$ of the function $f$ defined by (\ref{eq:spectrum}). We will use the following characterization of reflexive functions: 

\begin{proposition} \label{prop:HegalH0pourf}
The following are equivalent:
\begin{enumerate}
\item $\Spec_f (z)$ is a polynomial,
\item $\Spec_f (z) =\theta (z)$,
\item $H=H_0$,
\item $L_{\psi} (m)= L_{\psi}^0 (m)$ for $m\geq 1$, 
\item $L_{\psi}(-m)=(-1)^n L_{\psi} (m-1)$ for $m\geq 1$,
\item $\theta_p =\theta_{n-p}$ for $p=0,\ldots ,n$.
\end{enumerate}
\end{proposition}
\begin{proof} The equivalence between 1 and 3 follows from Lemma \ref{lemma:IsoGrFV}.
The equivalence between 2 and 3 follows from Remark \ref{rem:SpectreSym}.
The equivalence between 3 and 4 follows from the proof of Lemma \ref{lemma:EhrphiPol} and the equivalence between 4 and 5 follows from Theorem \ref{theo:EhrReciprocity}. Last the equivalence between 3 and 6 is precisely Corollary \ref{coro:SymDeltaH0}.
\end{proof}

\noindent Anti-reflexive functions are characterized similarly: in particular, by Proposition \ref{prop:DecompTheta}, $f$ is anti-reflexive if and only if $z^{n+1} \theta (z^{-1})= \theta (z)$.

\section{Coefficients and roots of Hodge-Ehrhart polynomials}
\label{sec:CoeffRoots}

In this section, $f$ denotes a tame regular function on a smooth affine complex variety of dimension $n$. Recall the polynomials 
\begin{equation}\nonumber
L^{\alpha}_{\psi} (X) =\theta_0^{\alpha}\binom{X+n}{n}+\theta_1^{\alpha}\binom{X+n-1}{n}+\ldots +\theta_{n-1}^{\alpha}\binom{X+1}{n}+\theta_n^{\alpha}\binom{X}{n},
\end{equation}
where 
$$\binom{X+i}{n}=\frac{1}{n!} (X+i)(X+i-1)\ldots (X+i-n+1),$$
 defined in Section 
\ref{sec:EhrPolSHM}. 
We will write
\begin{equation}\nonumber
L^{\alpha}_{\psi}(X)= c_n^{\alpha} X^n + c_{n-1}^{\alpha} X^{n-1}+\ldots + c_1^{\alpha} X+ c_0^{\alpha} 
\end{equation}
and
\begin{equation}\nonumber 
L_{\psi}(X)=\sum_{\alpha \in [0,1[} L^{\alpha}_{\psi} (X)= c_n X^n + c_{n-1} X^{n-1}+\ldots + c_1 X+ c_0 .
\end{equation}
We study here the coefficients and roots of these polynomials (as polynomials over $\cit$).
A first consequence of the reciprocity law is that about a half of these coefficients depend only on $H_0$:

\begin{proposition} \label{prop:CC0}
The coefficients of the polynomial $L_{\psi}$ satisfy the following properties:
\begin{enumerate}
\item $n! c_n =\mu$ where $\mu$ is the global Milnor of $f$,
\item  $c_{n-j}=c_{n-j}^0$ if $j$ is odd,
\item $(1-(-1)^j ) c_{n-j}=\sum_{\ell =0}^{j-1} (-1)^{-\ell +j+1} \binom{n-\ell}{n-j} c_{n-\ell}^0$
for $j=1,\ldots ,n$,
\item $c_0 =c_0^0 =\theta_0^0$,
\item $n! c_{n-1}=\frac{n}{2} \dim H_{0}$.
\end{enumerate}
\end{proposition}
\begin{proof}
We have $n! c_n =\theta_0 +\ldots +\theta_n =\dim H=\mu$.
By Theorem \ref{theo:EhrReciprocity}, we have
\begin{equation}\nonumber
L_{\psi}(X)-(-1)^n  L_{\psi} (-X) =L_{\psi}^0 (X)-(-1)^n L_{\psi}^0 (-X)
\end{equation}
and
\begin{equation}\nonumber
L_{\psi}(X)-(-1)^n  L_{\psi} (-X) = L_{\psi}^0 (X)-L_{\psi}^0 (X-1)
\end{equation}
The first equality gives 2 and the second one gives 3.                                                        
We have $c_0 =\theta_0 =\theta_0^0 =c_0^0$ because $\theta_0^{\alpha}=0$ if $\alpha\in ]0,1[$ by Lemma \ref{lemma:SymDelta}, whence 4.
For 5, note that
$$c_{n-1}=\frac{n}{2} c_n^0 = \frac{n}{2} \frac{ (\theta_0^0 +\ldots + \theta_n^0 )}{n!} =\frac{n}{2}\frac{ \dim H_{0}}{n!} $$
where the first equality follows from 3. 
\end{proof}

\begin{corollary}\label{coro:SumProduct}
 Let $z_1 ,\ldots , z_n$ be the roots of $L_{\psi}$. Then,
$$\sum_{i=1}^n z_i =-\frac{n}{2}\frac{\dim H_0}{\mu}\ \mbox{and}\ \prod_{i=1}^n z_i =(-1)^n n!\frac{\theta_0 }{\mu}.$$
\end{corollary}
\begin{proof}
Use Proposition \ref{prop:CC0} in order to get the values of $c_{n-1}/c_n$ and  $c_0/c_n$. 
\end{proof}

The next result shows that Hodge-Ehrhart polynomials may have integral roots: we will call these integral roots {\em trivial roots}: 

\begin{lemma}\label{lemma:IntegerRoots}
Assume that 
$\theta (z) =z^{\ell} U(z)$
 where $U(0)\neq 0$ and $\ell\geq 1$.
Then,
$$L_{\psi} (X) =(X+1-\ell )(X+2-\ell )\ldots (X+\ell -1)v(X)$$
where $v(\ell )\neq 0$.  
\end{lemma}
\begin{proof} 
By assumption, $\theta_{0}=\ldots =\theta_{\ell -1}=0$ and $\theta_{\ell }\neq 0$ therefore $L_{\psi} (0)=\ldots =L_{\psi}(\ell -1 )=0$ and $L_{\psi} (\ell )\neq 0$. The same results hold also for 
$L_{\psi}^{0}$ and we get $L_{\psi}(-i)=0$ for $i=1,\ldots , \ell -1$
using Theorem \ref{theo:EhrReciprocity}.
\end{proof}

If $f$ is reflexive ({\em resp.} anti-reflexive) 
we have 
\begin{equation}\label{eq:FunctionalEq}
L_{\psi} (-X)= (-1)^n L_{\psi} (X-1)\ (\mbox{{\em resp.}}\ L_{\psi}(-X)=(-1)^n  L_{\psi} (X))
\end{equation}
 by Proposition \ref{prop:HegalH0pourf}, 
hence the roots of $L_{\psi}$ are symmetrically distributed with respect to the ''critical line'' $\re z =-\frac{1}{2}$ ({\em resp.} $\re z =0$). It is natural to look more precisely for functions whose Hodge-Ehrhart polynomial have their roots on the lines $\re z=-1/2$ or $\re z= 0$.
The following terminology is suggested by \cite{Go}:

\begin{definition}\label{def:CLACL}
We will say 
that $L_{\psi}$ satisfies (CL) (resp. (CL*))
if its roots (resp. its non-trivial roots) are on the line $\re z =-\frac{1}{2}$ and that $L_{\psi}$ satisfies (ACL) (resp. (ACL*)) 
if its roots (resp. non-trivial roots) are on the line $\re z =0$.
\end{definition}

\begin{proposition}\label{prop:CNCL}
Assume that $L_{\psi}$ satisfies (CL) (resp. (ACL)). Then $f$ is reflexive (resp. anti-reflexive).	
\end{proposition}
\begin{proof} If $L_{\psi}$ satisfies (CL), the sum of its roots is equal to $-\frac{n}{2}$ and we 
get $\dim H_0 =\mu$, hence $H_0 =H$, from Corollary \ref{coro:SumProduct}. The remaining assertion is shown similarly.
\end{proof}

In order to precise Lemma \ref{lemma:IntegerRoots} and the converse of Proposition \ref{prop:CNCL}, we will use the
the following result which is due (up to a harmless shift) to Rodriguez-Villegas \cite{RV}:

\begin{lemma}\label{lemma:Kronecker}
Assume that 
$\theta (z) =z^k U(z)$	
where $U \in \zit [X]$ is a polynomial of degree $r>0$ whose roots are on the unit circle. Then,
\begin{equation}\nonumber
L_{\psi} (X) =(X+1-k)(X+2-k)\ldots (X+n-r-k) v(X)
\end{equation}
where $v$ is a polynomial of degree $r$ whose roots are on the line $\re z=-\frac{n+1-r}{2}+k$ 
if $r\leq n-1$ and $L_{\psi}$ is a polynomial of degree $n$ whose roots are on the line $\re z=-\frac{1}{2}$ 
if $r=n$.
\end{lemma}
\begin{proof}
First, notice that $\theta (1)=U(1)\neq 0$. Let us write
$$\frac{U( z)}{(1-z)^{n+1}}=\sum_{m\geq 0} L_{\psi}^1 (m) z^m .$$
Then, by \cite[p. 2251]{RV}, we have
$$L_{\psi}^1 (X) =(X+1)(X+2)\ldots (X+n-r) v^1 (X)$$
where the roots of the polynomial $v^1$ are on the line $\re z =-(n+1-r)/2$ 
(the product of linear factors before $v^1$ is equal to $1$ if $n=r$).
Thus,
$$\frac{z^k U( z)}{(1-z)^{n+1}}=\sum_{m\geq k} L_{\psi} (m) z^m $$
where $L_{\psi} (X)=L_{\psi}^1 (X-k)$ and the result follows.	
\end{proof}

\begin{theorem} 
\label{theo:Kron}
Assume that $\theta (z)= z^k U(z)$ where 
where $U \in \zit [X]$ is a polynomial of degree $r>0$ whose roots are on the unit circle.
\begin{enumerate}
\item If $k=0$, the roots of $L_{\psi}$ are on the line $\re z =-1/2$: $L_{\psi}$ satisfies (CL). 
\item If $k\geq 1$ and $f$ is reflexive we have 
\begin{equation}\nonumber
L_{\psi} (X) =(X+1-k)(X+2-k)\ldots (X+k) v(X)
\end{equation}
where $v$ is a polynomial of degree $n-2k$ whose roots are on the line $\re z=-\frac{1}{2}$: $L_{\psi}$ satisfies (CL*). 
\item If $k\geq 1$ and $f$ is anti-reflexive we have 
\begin{equation}\nonumber
L_{\psi} (X) =(X+1-k)(X+2-k)\ldots (X+k-1) v(X)
\end{equation}
where $v$ is a polynomial of degree $n-2k+1$ whose roots are on the line $\re z =0$: $L_{\psi}$ satisfies (ACL*).	

\end{enumerate}
\end{theorem}
\begin{proof}  
For 1, 
note first that, by assumption, $\theta_0 \neq 0$.
By Hodge symmetry, 
$\theta_n =\theta_n^0 +\theta_n^{\neq 0}= \theta_0^0 +\theta_n^{\neq 0}=\theta_0 +\theta_n^{\neq 0}\geq 1$ thus $\theta$ is a polynomial of degree $n$.
Then, we apply Lemma \ref{lemma:Kronecker} with $r=n$ and $k=0$ in order to get the assertion about the roots of $L_{\psi}$.
Assume that $f$ is reflexive. By Proposition
\ref{prop:HegalH0pourf} we have $\theta (z)=z^n \theta (z^{-1})$, therefore $z^{n-2k} U(z^{-1})=U(z)$. 
Due to the assumption on the roots of $U$, we have also $z^r U(z^{-1})= U(z)$. We conclude that $r=n-2k$ and 
we apply Lemma \ref{lemma:Kronecker} in order to get 2.
 In the same way, if $f$ is anti-reflexive we have 
$z^{n+1} \theta (z^{-1})=\theta (z)$ therefore $r=n-2k+1$, and 3 also follows from Lemma \ref{lemma:Kronecker}.\end{proof}

\begin{example}\label{ex:CLRU} 
The following examples illustrate the three cases of Theorem \ref{theo:Kron}.
\begin{enumerate}
\item Let $f$ be the Laurent polynomial on $(\cit^*)^6$ defined by
$$f(u_1 , u_2 , u_3 ,u _4 , u_5 , u_6 )=u_1 +u_2 +u_3 +u_4 +u_5+u_6 +\frac{1}{u_1 u_2 u_3 u_4 u_5 u_6^3}.$$
 Then $f$ is reflexive (by \cite{DoSa2} we have $\Spec_f (z)=1+z+2z^2+z^3+2z^4 +z^5 +z^6$) and
$\theta (z)=\Spec_f (z) = (1+z+z^2)(1+z^2 +z^4)$. It follows from Theorem \ref{theo:Kron} that $L_{\psi}$ satisfies (CL).
\item Let $n\geq 3$ and let $f$ be the polynomial on $\cit^n$ defined by
$$f(u_1 ,\ldots , u_n )=u_1 +\ldots +u_n +u_1 \ldots u_n.$$
 Then $f$ is reflexive (we have $\Spec_f (z)=z+z^2 +\ldots +z^{n-1}$) and
$\theta (z)=z+z^2 +\ldots +z^{n-1}$. It follows from  Theorem \ref{theo:Kron} that $L_{\psi}(X) =X(X+1)v (X)$
where $v$ is a polynomial of degree $n-2$ whose roots are on the line $\re z=-1/2$: $L_{\psi}$ satisfies (CL*).
\item Let $f$ be the polynomial on $\cit^2$ defined by
$f(u_1 , u_2 )=u_1^2 +u_2^3$.
 Then $f$ is anti-reflexive (since $\Spec_f (z)=z^{5/6}+z^{7/6}$) and $\theta (z)=z+z^2$. Theorem \ref{theo:Kron} asserts that $L_{\psi}(X) =Xv (X)$
where $v$ is a polynomial of degree $1$ whose roots are on the line $\re z=0$: $L_{\psi}$ satisfies (ACL*).
\end{enumerate}

\noindent Of course, the polynomial $L_{\psi}$ can be explicitly computed in the previous examples: for the first one, we check that
$L_{\psi} (X)=\frac{1}{6!}(9 X^6 +27X^5 +405 X^4 +765 X^3 +2106 X^2 +1728X+720)$ etc...
\end{example}

\begin{remark} 
One may wonder how common is the assumption of Theorem \ref{theo:Kron}. For instance, the roots of the $\theta$-vector of a Laurent polynomial whose Newton polytope is a reduced 
(in the sense of \cite{Conrads}) and reflexive three dimensional simplex are on the unit circle if and only if $\mu $ is equal to $4$, $6$ or $8$, that is in $4$ out of $14$ cases, see the list in loc. cit.
Nevertheless, we will see in the next section that their Thom-Sebastiani sums will provide an infinite number of functions satisfying the assumption of Theorem \ref{theo:Kron}.
\end{remark}

\begin{remark} Of course, the assumption of Theorem \ref{theo:Kron} is not necessary in order to get Hodge-Ehrhart polynomials satisfying (CL). For instance, let $f (u_1 ,u_2)=u_1 +u_2 +1/(u_1^2 u_2^3)$: the roots of its $\theta$-vector $\theta (z)= 1+4z+z^2$ are not on the unit circle but its Hodge-Ehrhart polynomial satisfies (CL).
\end{remark}

\begin{remark} Assume that $\theta_0 =1$ and that the roots of $\theta$ are on the unit circle. Then, 
the roots of $\theta$ are roots of unity.
This contains Stanley's description of Hilbert-Poincar\'e series of finitely generated graded algebras which are complete intersections \cite[Corollary 3.4]{Sta}. 
\end{remark}

\section{Thom-Sebastiani and Hodge-Ehrhart polynomials}
\label{sec:ThomSeb}

Let $f' :U\rightarrow \cit$ and $f'' :V\rightarrow\cit$ be two regular functions. We define their {\em Thom-Sebastiani sum}
$$f:= f' \oplus f'' : U\times V\rightarrow\cit$$
by 
$$f' \oplus f'' (u,v) =f' (u)+f'' (v).$$
 By \cite{NS}, $f$ is tame if $f'$ and $f''$ are tame. 
We will denote by $\theta (z)$ ({\em resp.} $\theta' (z)$, $\theta'' (z)$) the $\theta$-vector of $f$ ({\em resp.} $f'$, $f''$) and we will put $n_1 =\dim U$, $n_2 =\dim V$ and $n=n_1 +n_2$.

The following result shows that Thom-Sebastiani sums are useful in order to construct reflexive functions, anti-reflexive functions and functions satisfying the assumption of Theorem \ref{theo:Kron}:

\begin{theorem} 
\label{theo:ThomSeb}
Assume that the function $f'$ is reflexive.
Then
\begin{equation}\label{eq:TS0}
\theta (z)= \theta '(z)\theta '' (z)
\end{equation}
and the decomposition (\ref{eq:Decomposition}) is given by
\begin{equation}\label{eq:TS1}
\theta^0 (z)= \theta '(z)(\theta '')^0 (z)
\end{equation}
and
\begin{equation}\label{eq:TS2}
\theta^{\neq 0} (z)= \theta '(z)(\theta '')^{\neq 0} (z).
\end{equation}
In particular, $f$ is reflexive (resp. anti-reflexive) if and only if $f''$ is reflexive (resp. anti-reflexive).
\end{theorem}
\begin{proof} 
For the first assertion, note that 
$$\theta_i =\sum_{i-1<\beta\leq i}\nu_{\beta}=\sum_{i-1<\beta '+\beta '' \leq i}\nu_{\beta '}\nu_{\beta ''}=\sum_{\beta '} \nu_{\beta'}\sum_{i-1-\beta '<\beta '' \leq i-\beta '}\nu_{\beta ''}=\sum_{\beta '}\theta '_{\beta '}\theta ''_{i-\beta '}$$
where $\nu_{\beta}:=\dim \gr^V_{\beta} (G_0 /\theta G_0 )$ ({\em resp.} $\nu_{\beta '}$, $\nu_{\beta ''}$) denotes the multiplicity of $\beta$ ({\em resp.} $\beta '$, $\beta ''$) in the spectrum at infinity of $f$ ({\em resp.} $f'$, $f''$):
 the first equality follows from Remark \ref{rem:SpectreSym}, the second one from \cite[Proposition 3.7]{Sab} and the last one from the fact that $f'$ is reflexive (by Proposition \ref{prop:HegalH0pourf} $\beta '$ is a nonnegative integer and we have $\nu_{\beta'}=\theta '_{\beta '}$; note that formula (\ref{eq:TS0}) is visibly false without the reflexive assumption on $f'$).
Therefore, 
\begin{equation}\nonumber
\theta (z)= \theta '(z)\theta '' (z)= \theta '(z)(\theta '')^0 (z)+\theta '(z)(\theta '')^{\neq 0} (z)
\end{equation}
where $\theta '' (z)= (\theta '')^0 (z)+(\theta '')^{\neq 0} (z)$ is the decomposition given by
Proposition \ref{prop:DecompTheta}. 
Because $f'$ is reflexive, we get from Proposition \ref{prop:HegalH0pourf} 
\begin{equation}\nonumber
z^n \theta '(z^{-1})(\theta '')^0 (z^{-1})=z^{n_1}\theta '(z^{-1})z^{n_2}(\theta '')^0 (z^{-1})=\theta '(z)(\theta '')^{0} (z)
\end{equation}
and
\begin{equation}\nonumber
z^{n+1} \theta '(z^{-1})(\theta '')^{\neq 0} (z^{-1})=z^{n_1}\theta '(z^{-1})z^{n_2 +1}(\theta '')^{\neq 0} (z^{-1})=\theta '(z)(\theta '')^{\neq 0} (z).
\end{equation}
Thus, (\ref{eq:TS1}) and (\ref{eq:TS2})
follow from the unicity of the decomposition (\ref{eq:Decomposition}).
\end{proof}

\begin{example} \label{ex:ThomSeb}
In the following examples, we use Theorem \ref{theo:ThomSeb} in order to construct functions satisfying the assumption of Theorem \ref{theo:Kron}.  
\begin{enumerate} 
\item Let 
$$f' (u_1 ,\ldots, u_{n_1} )=u_1 +\ldots +u_{n_1} +\frac{1}{u_1 \ldots u_{n_1}}$$
and
$$f'' (v_1 ,\ldots, v_{n_2} )=v_1 +\ldots +v_{n_2} +\frac{1}{v_1 \ldots v_{n_2}},$$
defined respectively on $(\cit^*)^{n_1}$ and 
$(\cit^*)^{n_2}$. 
Then the $\theta$-vector of the Thom-Sebastiani sum $f=f' \oplus f''$ is
$$\theta (z) =\theta' (z)\theta'' (z)=(1+z+\ldots +z^{n_1})(1+z+\ldots +z^{n_2})$$
since, again by \cite{DoSa2} and Proposition \ref{prop:HegalH0pourf},
$\Spec_{f'} (z)=1+z+\ldots +z^{n_1}=\theta' (z)$ 
and 
$\Spec_{f''} (z)=1+z+\ldots +z^{n_2}=\theta'' (z)$ ($f'$ and $f''$ are reflexive). It follows from
Theorem \ref{theo:Kron} that the roots of the Hodge-Ehrhart polynomial of $f$ are on the line $\re z =-1/2$. 
\item

 Let 
$f' (u_1 ,u_2)=u_1 +u_2 +\frac{1}{u_1 u_2^2}$
and
$f'' (v_1 , v_{2} )=v_1^2 +v_{2}^3$,
defined respectively on $(\cit^*)^2$ and 
$\cit^2$. The function $f'$ is reflexive and the function $f''$ is anti-reflexive.
Thus $f=f' \oplus f''$ is anti-reflexive and its $\theta$-vector is
$$\theta (z) =\theta' (z)\theta'' (z)=(1+2z+z^2)(z+z^{2})=z(1+z)^3.$$
It follows from
Theorem \ref{theo:Kron} that the Hodge-Ehrhart polynomial of $f$ is $L_{f} (X)=Xv(X)$ where $v$ is a polynomial of degree $3$ whose roots are on the line $\re z =0$. 
\end{enumerate}
\end{example}

\begin{remark}
 Let $f=f' \oplus f''$ and assume that $f'$ is reflexive. Then, 
\begin{equation}\nonumber 
	h^{p,q} (f)=\sum_{i,j} h^{i,j} (f') h^{p-i, q-j} (f'')
\end{equation}
where $h^{p,q} (g)$ denotes the Hodge numbers of the function $g$. More precisely, we have
\begin{equation}\nonumber
	h^{p,q}_0 (f)=\sum_{i,j} h^{i,j} (f') h^{p-i, q-j}_0 (f'')\ \mbox{and}\ h^{p,q}_{\neq 0} (f)=\sum_{i,j} h^{i,j} (f') h^{p-i, q-j}_{\neq 0} (f'')
\end{equation}
where the subscripts refer to the decomposition $H=H_0 \oplus H_{\neq 0}$ (this is
an adaptation of \cite[Proposition 8.12]{ScSt}).
This could be used in order to construct regular functions with prescribed Hodge numbers. 
\end{remark}

\section{Sabbah's mixed Hodge structures and Ehrhart theory}

\label{sec:HodgeEhrhart}

We show in this section that the Hodge-Ehrhart theory developed in Section \ref{sec:DeltaSHM} is related to the classical Ehrhart theory for polytopes when $f$ is a Laurent polynomial. 

\subsection{Hodge and Ehrhart {\em via} Newton polytopes}
\label{sec:HodgeEhrhartNewtPol}

Recall that if
$P$ is a lattice polytope in $\rit^n$ (in what follows we will consider only full dimensional lattice polytopes), the function 
defined by $L_P (\ell ):= \card ( (\ell P )\cap N)$  for nonnegative integers $\ell$ is a polynomial in $\ell$ of degree $n$ and that we have
\begin{equation}\nonumber 
\Ehr_P (z):=1+\sum_{m\geq 1}L_P (m) z^m =\frac{\delta_0 +\delta_1 z +\ldots +\delta_n z^n}{(1-z)^{n+1}}
\end{equation}
where the $\delta_j$'s are nonnegative integers. The polynomial $L_P $ is called the Ehrhart polynomial of $P$, $\delta_P (z)=\delta_0 +\delta_1 z+ \ldots +\delta_n z^n$ is called the $\delta$-vector of $P$ and
the function $\Ehr_P$ is called the Ehrhart series of $P$ (see for instance \cite{BeckRobbins}).

\begin{theorem}\label{theo:ThetaEgalDelta}
Let $f$ be a convenient and nondegenerate Laurent polynomial on $(\cit^*)^n$ 
(in the sense of Kouchnirenko \cite{K})
and let $P$ be its Newton polytope. Assume that $P$ is simplicial. Then, 
$$\theta (z)=\delta_{P}(z)$$
where $\theta (z)$ is the $\theta$-vector of $f$ and $\delta_P (z)$ is the $\delta$-vector of $P$.
\end{theorem}
\begin{proof} Let $MHS_f =(H, F^{\bullet}, W_{\bullet})$ be the Sabbah's mixed Hodge structure of $f$.
We have to show that
$\dim\gr^p_F H=\delta_{n-p}$ 
for $p=0,\ldots ,n$.
Let $\alpha\in ]0,1[$. 
By \cite[Remark 4.8]{DoSa1} and 
because $f$ is convenient and nondegenerate, the Brieskorn module $G_0$ is a lattice in $G$ 
and
$$\dim\gr^p_F H_{\alpha}=\dim\gr^{V}_{\alpha +n-1-p} G_0 /\theta G_0  = \dim\gr^{\mathcal{N}}_{\alpha +n-1-p} G_0 /\theta G_0$$
where the first equality follows from Lemma \ref{lemma:IsoGrFV} and the second one follows from the identification between the $V$-filtration and the Newton filtration $\mathcal{N}$ given by \cite[Theorem 4.5]{DoSa1}. In the same way, 
$$\dim\gr^p_F H_{0}=\dim\gr^{V}_{n-p} G_0 /\theta G_0  = \dim\gr^{\mathcal{N}}_{n-p} G_0 /\theta G_0.$$
Because $n-1-p < \alpha +n-1-p < n-p$ if $\alpha \in ]0,1[$, the result now follows from \cite[Corollary 4.3]{D12}.
\end{proof}

 \noindent It should be emphasized that under the assumptions of Theorem \ref{theo:ThetaEgalDelta} the Newton polytope $P$ contains the origin as an interior point (this is the convenience assumption).

\begin{corollary}\label{coro:LegalL} 
Let $f$ be a convenient and nondegenerate Laurent polynomial and assume that its Newton polytope $P$ is simplicial.
Then, the Hodge-Ehrhart polynomial $L_{\psi}$ of the function $f$ is equal to the 
Ehrhart polynomial $L_{P}$ of the polytope $P$.\qed
\end{corollary}

\begin{remark} \label{rem:HegalH0}
The equivalent conditions
of Proposition \ref{prop:HegalH0pourf} hold for $f$ if and only if its Newton polytope $P$ is a reflexive lattice polytope (see \cite[Proposition 5.1]{D12}). 
\end{remark}

\subsection{Relations with prior results in combinatorics}

In this section, $P$ is a simplicial polytope containing the origin as an interior point, $f_P$ is a convenient and nondegenerate Laurent polynomial whose Newton polytope is $P$ (by \cite[Theorem 6.1]{K}, $f_P$ exists). By Section \ref{sec:HodgeEhrhartNewtPol}, we can apply the results known for the mixed Hodge structure $MHS_{f_P}$ in order to get informations about $P$, and vice versa. Here is a (non-exhaustive) overview.

\subsubsection{From singularity theory to combinatorics}
\label{subsec:EhrRecipFunctionPolytope}

\noindent {\em Special values.} Some coordinates of the $\delta$-vector $(\delta_0 ,\ldots , \delta_n )$ have an easy combinatorial description. 
The first point is that $\delta_0 =1$ which amounts to $\theta_0 =\dim\gr_F^n H =1$ by Theorem \ref{theo:ThetaEgalDelta}. This is precisely what gives Remark \ref{rem:theta0}. Also, $\theta_0 +\theta_1 +\ldots +\theta_n =\dim H=\mu$
where $\mu$ is the global Milnor number of $f$. By \cite{K}, $\mu =n! \vol(P)$ where the volume is normalized such that the volume of the cube is equal to one and this agrees with the classical formula for
$\delta_0 +\delta_1 +\ldots +\delta_n$. We refer to \cite[Proposition 2.6]{D12} for a discussion about the equalities 
$\theta_n =\delta_n =\Card ( (P -\partial P)\cap \zit^n )$
and $\theta_1 =\delta_1 =\Card (P\cap\zit^n )-n-1$.\\

\noindent {\em Hodge symmetry and Hibi's palindromic theorem.} By Theorem \ref{theo:ThetaEgalDelta} and Corollary \ref{coro:SymDeltaH0}, we have $\delta_{n-p} =\delta_{p}$ for all $p$
if and only if $H=H_0$, that is if and only if $P$ is a reflexive polytope by Remark \ref{rem:HegalH0}. This agrees with Hibi's palindromic theorem \cite{Hibi}.\\

\noindent {\em Hodge-Ehrhart reciprocity and Ehrhart reciprocity.}
We show now that Ehrhart reciprocity for $P$ can be deduced from  the reciprocity law  given by Theorem \ref{theo:EhrReciprocity} for $f_P$.
By Remark \ref{rem:SpectreSym} and \cite[Theorem 4.1]{D12}, we have
\begin{equation}\nonumber
\theta^0 (z) =(1-z)^{n+1} \sum_{m\geq 0} L_{P}^0 (m) z^m
\end{equation}
where $L_{P}^0 (0)=1$ and $L_{P}^0 (m)$ is equal to the number of lattice points on $\partial P \cup  \partial (2 P)\cup \ldots \cup \partial (m P)$ plus one for $m\geq 1$ ($\partial Q$ denotes the boundary of $Q$). In particular,
$L_{\psi}^0 (m) =L_{P}^0 (m)$    
for $m\geq 0$ and Theorem \ref{theo:EhrReciprocity}, together with Corollary \ref{coro:LegalL}, provides 
\begin{equation}\nonumber
L_{P}(-m)=(-1)^n ( L_{P} (m) +L_{P}^0 (m-1)-L_{P}^0 (m))
\end{equation}
for $m\geq 1$. Because $L_{P}^0 (m)-L_{P}^0 (m-1)$ is equal to the number of lattice points on $\partial (mP)$, we finally get, for $m\geq 1$,
\begin{equation}\nonumber
L_{P}(-m)=(-1)^n L_{P^{\circ}} (m) 
\end{equation}
where $P^{\circ}$ denotes the interior of $P$. This is the classical Ehrhart reciprocity.

More generally, and via Theorem \ref{theo:ThetaEgalDelta},
the polynomials $L_{\psi}^{\alpha}$ correspond (up to a shift) to the weighted Ehrhart polynomials $f_k^{0}$
defined in \cite{Stapledon}. In this setting, our Theorem \ref{theo:EhrReciprocity} corresponds to the ``Weighted Ehrhart
Reciprocity'' given by Theorem 3.7 of {\em loc. cit.}\\

\noindent {\em Coefficients and roots.} The results in Section \ref{sec:CoeffRoots}, when
applied to Laurent polynomials, agree with well-known facts in combinatorics (the basics can be found in \cite{BeckRobbins}, see also \cite{BDDPS}, \cite{BHW}...). In this setting, condition (CL) for Ehrhart polynomials of reflexive polytopes is analyzed in details in \cite{HHK}, where a complete list of references about the subject is also given.\\

\noindent {\em Free sums of polytopes.}  
If $P$ is a reflexive polytope and $Q$ is a polytope containing the origin as an interior point, it is shown in \cite{Braun} that
$$\Ehr_{P\oplus Q} (z) =(1-z) \Ehr_{P}(z)\Ehr_{Q} (z)$$
where $P\oplus Q$ denotes the free sum of the polytopes $P$ and $Q$.
This is what gives Theorem \ref{theo:ThomSeb} when applied to convenient and nondegenerate Laurent polynomials with simplicial Newton polytopes $P$ and $Q$.\\

\noindent {\em Decomposition of the $\delta$-vector.}
According to \cite[Theorem 5]{BM}, the $\delta$-vector of a lattice polytope containing the origin as an interior point has a decomposition $\delta (z)= A(z)+zB(z)$ where 
$\deg A(z)=n$, $\deg B(z)\leq n-1$, $A(z)=z^n A(z^{-1})$
and $B(z)=z^{n-1} B(z^{-1})$. This is precisely what gives 
Proposition \ref{prop:DecompTheta} (because $\theta^{\neq 0} (0)=0$, we may write 
$\theta^{\neq 0} (z)= z \overline{\theta} (z)$ where $\overline{\theta} (z)= z^{n-1}\overline{\theta}(z^{-1})$).

\subsubsection{From combinatorics to singularity theory: the Lower Bound Theorem}

Our last point is the counterpart of a result in combinatorics proven by Hibi in \cite{Hibi1} and known as the ``Lower Bound Theorem'':

\begin{proposition}\label{prop:LowerBoundTheorem}
Let $f$ be a convenient and nondegenerate Laurent polynomial and let $MHS_f =(H, F^{\bullet}, W_{\bullet})$ be its Sabbah's mixed Hodge structure. Assume that its Newton polytope $P$ is simplicial. Then,
\begin{equation} \nonumber 
1\leq \dim\gr_F^{n-1}H\leq \dim\gr_F^{n-i}H
\end{equation} 
for $1\leq i\leq n-1$.
\end{proposition}
\begin{proof}
Let $P$ be the Newton polytope of $f$ and let 
$\delta_P (z)=\delta_0 +\delta_1 z +\ldots +\delta_n z^n$ be its $\delta$-vector.
First, $\dim \gr_F^{n-1}H =\delta_1 =\Card (P\cap\zit^n )-n-1 \geq 1$ because $P$ contains the origin as an interior point. Now, we have $\delta_n =\Card ( (P -\partial P)\cap \zit^n )\geq 1$ (again because $P$ contains the origin as an interior point) and therefore, by \cite{Hibi1}, we get
$\delta_1\leq \delta_i \ \mbox{for}\ 1\leq i\leq n-1$. 
Now, the result follows from Theorem \ref{theo:ThetaEgalDelta} 
\end{proof}

\section{Linear inequalities among the Hodge numbers of Sabbah's mixed Hodge structures}

\label{sec:LinearIn}

Let $f$ be a tame regular function on a smooth complex affine variety of dimension $n$. Recall the decomposition of its $\theta$-vector
$\theta (z) =\theta^0 (z) +\theta^{\neq 0} (z)$
given by Proposition \ref{prop:DecompTheta}.
We write $\theta^0 (z)=\sum_{i=0}^n \theta_i^0 z^i$ and $\theta^{\neq 0} (z)=\sum_{i=0}^n \theta_i^{\neq 0} z^i$. After Proposition \ref{prop:LowerBoundTheorem}, we discuss in this section some linear inequalities involving the coefficients $\theta_i$, $\theta_i^{0}$ and $\theta_i^{\neq 0}$. We will make a repeated use of the symmetries
\begin{equation}\label{eq:Sym2}
\theta_{p}^{0}=\theta_{n-p}^{0}
\end{equation}
and
\begin{equation}\label{eq:Sym3}
\theta_{p}^{\neq 0}=\theta_{n-p+1}^{\neq 0}
\end{equation}
for $p=0,\ldots ,n$ (see Section \ref{sec:EhrSHM}).

\subsection{The $H$-conjecture}

\label{sec:ConjH}

We begin with the simplest linear inequality among the coefficients of the $\theta$-vector:

\begin{lemma}\label{lemma:ThetanTheta0}
We have $\theta_n\geq \theta_0$.
\end{lemma}
\begin{proof}
Indeed, we have $\theta_n =\theta^0_{n} +\theta^{\neq 0}_{n} \geq \theta^0_{n} =\theta^0_{0} =\theta_{0}$, where the penultimate equality follows from (\ref{eq:Sym2}) and the last equality follows from Lemma \ref{lemma:Ffinite}.
\end{proof}

\noindent  Despite its simplicity, this first linear inequality is already significant, as shown by the next result: 

\begin{proposition}\label{prop:DeltaThetaIFF}
The $\delta$-vector of a simplicial polytope $P$ is a $\theta$-vector if and only if 
$P$ contains the origin as an interior point.
\end{proposition}
\begin{proof} By Theorem \ref{theo:ThetaEgalDelta}, the $\delta$-vector of a $n$-dimensional simplicial lattice polytope containing the origin as an interior point is always a $\theta$-vector. It follows from 
Lemma \ref{lemma:ThetanTheta0}
that the $\delta$-vector of a simplicial lattice polytope $P$ without interior lattice points is never a $\theta$-vector since in this case $\delta_0 =1$ and $\delta_n =0$ ($\delta_n$ is equal to the number of lattice points in the interior of $P$).
\end{proof}

We will be mainly interested in the following other linear inequalities: 
inequality (\ref{eq:H}) is suggested by \cite{Hibi1} and inequalities (\ref{eq:H++}) and (\ref{eq:Douai}) are natural variations of it:

\begin{definition}\label{def:H}
Assume that $n\geq 3$. We will say that 
\begin{enumerate}
\item $f$ satisfies the H-property if 
\begin{equation}\label{eq:H}
\theta_i \geq \theta_1	\geq\theta_n
\end{equation}
 for $i=1,\ldots , n-1$,	
\item $f$ satisfies the $H^{\circ}$-property  if 
\begin{equation}\label{eq:H++}
\theta_{i}^0 \geq \theta_1^0
\end{equation}
for $i=1,\ldots , n-1$,
\item $f$ satisfies the $H^*$-property if 
\begin{equation}\label{eq:Douai}
\theta_i^{\neq 0} \geq \theta_1^{\neq 0}
\end{equation}
for $i=1,\ldots , n-1$.	
\end{enumerate}
\end{definition}

\noindent If $n=2$, only the condition $\theta_1 \geq \theta_2$ is relevant.

\begin{remark} 
\label{rem:Symmetry} 
It follows from the symmetry properties
(\ref{eq:Sym2}) and (\ref{eq:Sym3}) that
(\ref{eq:H++}) is equivalent to 
\begin{equation}\nonumber
\theta_{n-1}+\theta_{n-2}+\ldots +\theta_{n-i} \leq \theta_2 +\theta_3 +\ldots +\theta_{i+1}
\end{equation}
for $i=1,\ldots ,[(n-1)/2]$ 
(this inequality could have been suggested by \cite{Hibi1} and Theorem \ref{theo:ThetaEgalDelta})
and that (\ref{eq:Douai}) is equivalent to 
\begin{equation}\nonumber 
	\theta_{n-1}+\theta_{n-2}+\ldots +\theta_{n-i} \geq \theta_1 +\theta_2 +\ldots +\theta_{i}
\end{equation}
for $i=1,\ldots ,[(n-1)/2]$.	
\end{remark}

\begin{remark}
Other linear inequalities could be suggested by combinatorics:
thanks to Theorem \ref{theo:ThetaEgalDelta}, and
 after Hibi \cite[Theorem A]{Hibi0}, one could also consider 
the inequality 
\begin{equation}\nonumber 
\theta_{0}+\theta_{1}+\cdots +\theta_{i+1} \geq \theta_n +\theta_{n-1} +\cdots\theta_{n-i}\ \mbox{for}\ i=0,\cdots , [(n-1)/2].	
\end{equation}
But it always holds for $\theta$-vectors since it is equivalent to $\theta_{i+1}^{0}\geq 0$. 
 After Stanley \cite[Proposition 4.1]{Stan1}, one could also consider 
the inequality
\begin{equation}\nonumber 
\theta_{0}+\theta_{1}+\cdots +\theta_{i} \leq \theta_n +\theta_{n-1} +\cdots+\theta_{n-i}\ \mbox{for}\ i=0,\cdots ,n
\end{equation}
if $\theta_n \neq 0$. 
But it always holds for $\theta$-vectors since it is equivalent to $\theta_{n-i}^{\neq 0}\geq 0$. 
\end{remark}

In general, we expect that the properties of Definition \ref{def:H} are always satisfied:

\begin{conjecture}\label{conj:H}
($H$-Conjecture) A tame regular function $f$ on a smooth affine complex variety of dimension $n\geq 3$ satisfies the $H$-property, the $H^{\circ}$-property and the $H^*$-property.
\end{conjecture}

\noindent Two main motivations are in order:  the $H$-conjecture predicts informations about the distribution of the spectrum at infinity of the function $f$ since (\ref{eq:H++}) ({\em resp.} (\ref{eq:Douai})) means that  the number of spectral numbers equal to $i$ ({\em resp.} contained in 
	$]i-1,i[$) is greater or equal than the number of spectral numbers equal to $1$ ({\em resp.} contained in $]0,1[$); if true, it would be also useful in order to test if a polynomial is a $\theta$-vector or not (as in Proposition \ref{prop:DeltaThetaIFF}). See also Section \ref{sec:Distrib} below.

\begin{remark}
 Recall the Hodge numbers $h^{p,q}$ of the mixed Hodge structure $MHS_f$, see 
(\ref{eq:Defhpq}). Since $\theta_p =\dim \gr_F^{n-p} H= \sum_{q}h^{n-p,q}$,
the $H$-conjecture predicts linear inequalities between these Hodge numbers: for instance, the inequality $\theta_i \geq \theta_1$ for $i=1,\ldots , n-1$ is equivalent to 
\begin{equation}\nonumber 
\sum_q h^{n-i,q}\geq \sum_q h^{n-1,q}
\end{equation}
for $i=1,\ldots ,n-1$.  One may wonder to what extent these inequalities hold 
for the Hodge numbers of compact K\"ahler manifolds. 
\end{remark}

\subsection{The $H$-conjecture: some positive answers} 
\label{sec:HT}

\noindent {\bf {\em The $H$-conjecture and Laurent polynomials}}.
The first evidence is provided by Laurent polynomials $f$ whose Newton polytopes contain the origin as an interior point:

\begin{proposition}
\label{prop:HibiPolLaur} 
Let $f$ be a convenient and nondegenerate Laurent polynomial defined on $(\cit^*)^n$. Assume that its Newton polytope $P$ is simplicial. Then $f$ satisfies the $H$-property and the $H^{\circ}$-property. 
\end{proposition}
\begin{proof} 
First, we have $\theta_1 =\theta^0_1 +\theta^{\neq 0}_1$ and $\theta_n =\theta^0_n +\theta^{\neq 0}_n =\theta^0_0 +\theta^{\neq 0}_1 =1+ \theta^{\neq 0}_1$ by 
the symmetry properties
(\ref{eq:Sym2}) and (\ref{eq:Sym3}). Since $\theta_1^0 =\Card (\partial P \cap \zit^n )-n$ (see for instance \cite{D12}), we have $\theta_1^0 \geq 1$ (because $P$ has at least $n+1$ vertices) and we finally get $\theta_1 \geq \theta_n$. 
Using moreover Proposition \ref{prop:LowerBoundTheorem}, we see that
$f$ satisfies the $H$-property. It follows from Remark \ref{rem:Symmetry}, \cite[Remark 1.4]{Hibi} and Theorem \ref{theo:ThetaEgalDelta}
that $f$ satisfies also the $H^{\circ}$-property.
\end{proof}

\begin{remark}
 Let $f$ be a tame polynomial defined on $\cit^n$.
The main difference with the case discussed above is that there is a priori no 
link between Hodge theory and Ehrhart theory in this situation. More precisely, the $\theta$-vector of a convenient and nondegenerate polynomial is never equal to the $\delta$-vector of its Newton polytope (the latter is determined by a ''toric''  spectrum, see \cite{D12}) since $\theta_0 =0$ and $\delta_0 =1$. The $H$-conjecture for the $\theta$-vector of a polynomial function
is still an open problem. As an example,
let us consider the polynomial $f$ defined on $\cit^3$ by
$f(u_1 ,u_2 , u_3 )=u_1^3 +u_2^3 +u_3^3 $.
Then,
$\Spec_{f} (z)=z+3z^{4/3}+3z^{5/3}+z^2$ and we get
$$\theta (z)= \theta^0 (z) +\theta^{\neq 0} (z)=(z+z^2) +6z^2=z+7z^2.$$
Therefore, $f$ satisfies the $H$-conjecture. This example is interesting because the $\delta$-vector of the Newton polytope $P$ of $f$ is $\delta_P (z)= 1+16z+10z^2$, which does not satisfy the $H$-property: this is not in contradiction with the $H$-conjecture since, by Lemma \ref{lemma:ThetanTheta0}, $\delta_P (z)$ cannot be the $\theta$-vector of a regular function.
\end{remark}

\noindent {\bf {\em The $H$-conjecture and unimodality}}.
Let $f$ be a regular tame function defined on an affine manifold of dimension $n$ and let $\theta (z)$ be its $\theta$-vector.
Recall that a polynomial $P (z)=a_0 +a_1 z+\ldots +a_n z^n$ is {\em unimodal} if 
$$a_0 \leq a_1 \leq \ldots \leq a_j \geq a_{j+1}\geq \ldots \geq a_n$$
 for some $0\leq j\leq n$.

\begin{proposition}\label{prop:UnimH}
Assume that $\theta^0 (z)$ and $\theta^{\neq 0} (z)$ are unimodal. Then $f$ satisfies the $H$-conjecture.
\end{proposition}
\begin{proof}
It follows from the symmetry property (\ref{eq:Sym2}) that $f$ has the $H^{\circ}$-property and that $\theta_1\geq \theta_n$ if $\theta^0$ is unimodal. By (\ref{eq:Sym3}), $f$ has also the $H^{*}$-property if $\theta^{\neq 0}$ is unimodal. Because $\theta_i =\theta_i^0 + \theta_i^{\neq 0}$, we get $\theta_i \geq \theta_1$ for $i=1,\ldots ,n-1$.
\end{proof}

A lattice polytope is {\em smooth Fano} if it contains the origin as an interior point and if the vertices of each facet form a lattice basis of $\zit^n$.

\begin{corollary}
Assume that the Newton polytope $P$ of the (convenient and nondegenerate) Laurent polynomial $f$ is smooth Fano. Then $f$  satisfies the $H$-conjecture.
\end{corollary}
\begin{proof} We may assume that $f (u) =\sum_{b\in \cal{V} (P)} u^b$ where $\cal{V}(P)$ denotes the set of the vertices of $P$.
If $P$ is smooth Fano, we have $\theta_p =\theta_p^0 =\dim H^{2p} (X_{\Sigma}, \cit)$ where 
$X_{\Sigma}$ denotes the toric variety of the fan $\Sigma$ over the faces of $P$ (see \cite{D12}).  
By the hard Lefschetz theorem for $X_{\Sigma}$, it follows that the $\theta$-vector of $f$ is unimodal and we apply Proposition \ref{prop:UnimH}.
\end{proof}

Recall that the mixed Hodge structure $MHS_f$ is of {\em Hodge-Tate type} if its Hodge numbers satisfy $h^{p,q}=0$ for $p\neq q$ (see Remark \ref{rem:HodgeTate}).

\begin{corollary}
Assume that the mixed Hodge structure $MHS_f$ is of Hodge-Tate type. Then $f$ satisfies the $H$-conjecture.	
\end{corollary}
\begin{proof} 
Since the mixed Hodge structure is of Hodge-Tate type, we have $\gr_F^p  \gr^W_{2q} H_{\alpha} =0$ for $p\neq q$ therefore the isomorphisms (see the proof of Proposition \ref{prop:SymHodgeNumbers})
\begin{equation}\nonumber
N^{2p-(n-1)} : \gr_F^p  \gr^W_{2p} H_{\alpha} \stackrel{\cong}{\longrightarrow}  \gr_F^{n-1-p}  \gr^W_{2(n-1)-2p}H_{\alpha}
\end{equation}
for $\alpha \in ]0,1[$ and $2p\geq n-1$ and 
\begin{equation}\nonumber
N^{2p-n} : \gr_F^p  \gr^W_{2p} H_{0} \stackrel{\cong}{\longrightarrow}  \gr_F^{n-p}  \gr^W_{2n-2p}H_{0}
\end{equation}
for $2p\geq n$. Hence, $h^{p,p}_0\geq h^{p-1,p-1}_0\ \mbox{if}\ 2p\leq n$ and 
$h^{p+1,p+1}_0\leq h^{p,p}_0\ \mbox{if}\ 2p\geq n$. In the same way, $ h^{p,p}_{\neq 0}\geq h^{p-1,p-1}_{\neq 0}$ if $2p\leq n-1$ and 
$h^{p+1,p+1}_{\neq 0}\leq h^{p,p}_{\neq 0}\ \mbox{if}\ 2p\geq n-1$ 
 (the subscripts refer to the decomposition $H=H_0 \oplus H_{\neq 0}$).
It follows that the $\theta$-vectors $\theta^0 (z)$ and $\theta^{\neq 0} (z)$ are unimodal (because $\theta_p =h^{n-p, n-p}$ by the Hodge-Tate assumption) 
and we apply Proposition \ref{prop:UnimH}.	
\end{proof}

\begin{remark} \label{rem:HLConstr} 
If $MHS_f$ is of Hodge-Tate type, the $H$-conjecture therefore follows from the Lefschetz conditions and the symmetry of the Hodge numbers.
If $MHS_f$ is not necessarily of Hodge-Tate type, we have the Lefschetz conditions
$$h^{p,q}_0\geq h^{p-1,q-1}_0\ \mbox{if}\  p+q\leq n,\ h^{p+1,q+1}_0\leq h^{p,q}_0\ \mbox{if}\ p+q\geq n$$
 and
 $$h^{p,q}_{\neq 0}\geq h^{p-1,q-1}_{\neq 0}\ \mbox{if}\  p+q\leq n-1,\ h^{p+1,q+1}_{\neq 0}\leq h^{p,q}_{\neq 0}\ \mbox{if}\ p+q\geq n-1.$$
One may wonder if the $H$-conjecture still follows from these Lefschetz conditions and the symmetry of the Hodge numbers, in other words if the equalities predicted by the $H$-conjecture are in some sense trivial or not (compare with \cite{Si}).
\end{remark}

\noindent {\bf {\em The $H$-conjecture and Thom-Sebastiani sums.}} It turns out that the $H$-conjecture behaves very well under the Thom-Sebastiani sums defined in Section \ref{sec:ThomSeb}: 

\begin{proposition} 
Let $f'$ be a reflexive function having the $H$-property.
\begin{enumerate}
\item If $f''$ has the $H$-property and the $H^{\circ}$-property then the Thom-Sebastiani sum $f=f'\oplus f''$ has the $H$-property and the $H^{\circ}$-property. 
\item 	If $f''$ has the $H$-property and the $H^*$-property then the Thom-Sebastiani sum $f=f'\oplus f''$ has the $H$-property and the $H^*$-property. 
\end{enumerate}
\end{proposition}
\begin{proof} 
Assume that $f'$ is reflexive and that $f'$ and $f''$ have the $H$-property.  
By (\ref{eq:TS0}), we have
$\theta_i =\sum_{p+q=i} \theta'_p \theta''_q $.
Let $i= n_1 +n_2 -j$ where $j\geq1$ (recall that $n_1 ,n_2\geq 3$). Then, the sum on the right contains 
$\theta'_{n_1-k} \theta''_{n_2 -\ell }+ \theta'_{n_1 -p} \theta''_{n_2 -q}$ where $k+\ell =p+q=j$, $(k,\ell)\neq (p,q)$ and $1\leq k\leq n_1 -1$, $0\leq \ell\leq n_2$, $0\leq p\leq n_1$ and $1\leq q\leq n_2 -1$.
Since $f'$ and $f''$ have the $H$-property, we have $\theta'_{n_1-k}\geq \theta'_1$ (resp. $\theta''_{n_2-q}\geq \theta''_1$). 
Moreover, by Lemma \ref{lemma:ThetanTheta0}, $\theta''_{n_2 -\ell}\geq 
\theta''_{0}$ (resp. $\theta'_{n_1 -p}\geq \theta'_{0}$).  
Therefore,
$$\theta_i =\sum_{p+q=i} \theta'_p \theta''_q \geq \theta'_1 \theta''_0 +\theta'_0 \theta''_1 =\theta_1 $$
for $i=1, \ldots ,n-1$.
In order to show that $f$ has the $H$-property, it remains to show that $\theta_1 \geq \theta_n$ (equivalently, $\theta_1^0 \geq \theta_0^0$). 
Using (\ref{eq:TS1}), we get
$$\theta_1^0 = \theta'_1 (\theta'')^0_0 +\theta'_0 (\theta'')^0_1 \geq \theta'_0 (\theta'')^0_0 = \theta_0^0$$
because $\theta'_1\geq \theta'_{n_1}\geq \theta'_0$, and this gives the required inequality.
In the same way, if $f''$ has the $H^{\circ}$-property we get, using again (\ref{eq:TS1}),
$$\theta_i^0 =\sum_{p+q=i} \theta'_p (\theta'')^0_q \geq \theta'_1 (\theta'')^0_0 +\theta'_0 (\theta'')^0_1 =\theta_1^0 $$
for $i=1, \ldots ,n-1$.
Thus, $f$ has the $H^{\circ}$-property and the first point follows. The second one is shown similarly, taking into account (\ref{eq:TS2}).
\end{proof}

\begin{corollary}
Assume that $f'$ is reflexive and that it satisfies the $H$-conjecture. Assume moreover that $f''$ satisfies the $H$-conjecture. Then the Thom-Sebastiani sum $f'\oplus f''$ satisfies the $H$-conjecture.\qed
\end{corollary}

\subsection{Application to the $\theta$-vector of a Laurent polynomial}

\label{sec:Distrib}

Last, we give some consequences for the $\theta$-vector of functions $f$ satisfying the $H$-property and the $H^{\circ}$-property (by Proposition
\ref{prop:HibiPolLaur}, this is the case if $f$ is a convenient and nondegenerate Laurent polynomial).

\begin{lemma}
Assume that $f$ satisfies the $H$-property and the $H^{\circ}$-property. Then, if $n\geq3$, we have
\begin{enumerate}
\item $\theta_i \geq \theta_1 \geq \theta_n \geq \theta_0$ for $i=1,\ldots , n-1$,	
\item $\theta_2^{\neq 0} =\theta_{n-1}^{\neq 0} \geq \theta_1^{\neq 0}$,
\item $\theta_2 \geq \theta_{n-1}$, 
\item $\theta^0_{i} \geq \theta_1^0$ for $i=1,\ldots , n-1$.
\end{enumerate} 
\end{lemma}
\begin{proof} 
The first point follows from the definition of the $H$-property and Lemma \ref{lemma:ThetanTheta0}.
The second point follows from (\ref{eq:H}) with $i=n-1$, together with Proposition 
\ref{prop:DecompTheta}. The third point follows from (\ref{eq:H++}) with $i=2$ and the last one is the definition of the $H^{\circ}$-property.
\end{proof}

\noindent Therefore, we get $\theta_2 \geq \theta_1 \geq \theta_3\geq \theta_0$ if $n=3$ and 
$\theta_2 \geq \theta_3\geq \theta_1 \geq \theta_4\geq \theta_0$ if $n=4$.

\begin{proposition}
Assume that $f$ satisfies the $H$-property and the $H^{\circ}$-property. Then,
\begin{enumerate}
\item  $\theta (z)$ is unimodal if $n\leq 4$, 
\item $\theta^0 (z)$ is unimodal if $n\leq 5$,
\item  $\theta^{\neq 0} (z)$ is unimodal if $n\leq 4$,
\item $f$ has also the $H^*$-property if $n\leq 4$.
\end{enumerate}
\end{proposition}
\begin{proof}
This is a direct consequence of the previous lemma, using the symmetry properties 
(\ref{eq:Sym2}) and (\ref{eq:Sym3}).
\end{proof}

\noindent If $n\geq 5$, we cannot expect something better than 
$$\theta_{2}\geq \theta_{n-1}\ \mbox{and}\ 
\theta_{i}\geq \theta_1 \geq \theta_n\geq \theta_0\ \mbox{for}\ i=1,\ldots, n-1.$$ 
For instance, if $n=5$, the position of $\theta_3$ with respect to $\theta_2$ and $\theta_4$ may vary: indeed, let us consider the Laurent polynomial $f$ 
defined by
$$f(u_1 ,\ldots ,u_5 )=u_1+u_2 +u_3 +u_4 +u_5 +\frac{1}{u_1 u_2 u_3 u_4 u_5^s}$$
where $s$ is a positive integer. Then $\mu =s+5$.
\begin{itemize}	
\item If $s=2$, we have $2=\theta_3 > \theta_2 =\theta_4 =1$. The $\theta$-vector of $f$ is $\theta (z)= 1+z+z^2+2z^3 +z^4+ z^5$ is unimodal. Note that $\theta (z)=\theta^0 (z)+\theta^{\neq 0} (z)$ where 
$\theta^0 (z)=1+z+z^2 +z^3 +z^4 +z^5$ and $\theta^{\neq 0} (z)=z^3$.
\item If $s=3$, we have $2=\theta_4 = \theta_2 >\theta_3 =1$. The $\theta$-vector of $f$ is $\theta (z)= 1+z+2z^2+z^3 +2z^4+ z^5$ is not unimodal. Moreover, $\theta (z)=\theta^0 (z)+\theta^{\neq 0} (z)$ where 
$\theta^0 (z)=1+z+z^2 +z^3 +z^4 +z^5$ and $\theta^{\neq 0} (z)=z^2 +z^4$.
\end{itemize} 
This example also shows that the unimodality of $\theta^{\neq 0} (z)$ and $\theta (z)$ is
no longer true if $n\geq 5$.


\begin{thebibliography}{999} 
\bibitem[1]{BeckRobbins} Beck, M., Robbins, S.: {\em Computing the continuous discretely}, Springer, New York (2007).
\bibitem[2]{BDDPS} Beck, M., De Loera, J. A., Develin, M., Pfeifle, J., Stanley, R. P.: {\em Coefficients and roots of Ehrhart polynomials}, Cont. Math., {\bf 374}, p. 15-36, 2005.
\bibitem[3]{BHW} Bey, C.,  Henk, M.,  Wills, J.: {\em Notes  on  the  roots  of  Ehrhart polynomials}, Discrete Comput. Geom., {\bf 38} (1), p. 81-98, 2007.
\bibitem[4]{BM} Betke, U., McMullen, P.: {\em Lattice points in lattice polytopes}, Mh. Math., {\bf 99}, p. 253-265, 1985.
\bibitem[5]{Braun} Braun, B.: {\em An Ehrhart series formula for reflexive polytopes}, Electron. J. Combin. 13, {\bf 1}, 2006.
\bibitem[6]{Conrads} Conrads, H.: {\em Weighted projective spaces and reflexive simplices}, Manuscripta Math., {\bf 107}, p. 215-227, 2002.
\bibitem[7]{DK} Danilov, V. I., Khovanskii, A. G.: {\em Newton polyhedra and an algorithm for calculating Hodge-Deligne numbers}, 
Izv. Akad. Nauk SSSR Ser. Mat., {\bf 50} (5), p. 925-945, 1986. 
\bibitem[8]{Del} Deligne, P.: {\em Th\'eorie de Hodge : II}, Pub. Math. IHES, {\bf 40}, p. 5-57, 1971.
\bibitem[9]{D13} Douai, A.: {\em Hard Lefschetz properties and distribution of spectra in singularity theory and Ehrhart theory}, J. Singul., {\bf 23}, p. 116-126, 2021. 
\bibitem[10]{D12} Douai, A.: {\em Ehrhart polynomials of polytopes and spectrum at infinity of Laurent polynomials}. J. Alg. Comb. (2020), {\bf 54},  p. 719-732, 2021. 
\bibitem[11]{DoMa} Douai, A., Mann, E.: {\em The small quantum cohomology of a weighted projective space, a mirror D-module and their classical limits}, Geom. Dedicata, {\bf 164}, p. 187-226, 2013.
\bibitem[12]{DoSa1} Douai, A., Sabbah, C.: {\em Gauss-Manin systems, Brieskorn
lattices and Frobenius structures I}, Ann. Inst. Fourier, {\bf 53} (4), p. 1055-1116, 2003.
 \bibitem[13]{DoSa2} Douai, A., Sabbah, C.: {\em Gauss-Manin systems, Brieskorn lattices and Frobenius structures II}, In : Frobenius Manifolds, C. Hertling and M. Marcolli (Eds.), Aspects of Mathematics E 36, 2004.
\bibitem[14]{Go} Golyshev, V.: {\em On the canonical strip}, Russ. Math. Surv., {\bf 64}, p.  145-147, 2009.
\bibitem[15]{HHK} Hegedus, G., Higashitani, A., Kasprzyk, A.: {\em Ehrhart polynomial roots of reflexive polytopes}, Elec. Journal of combinatorics,  {\bf 26}, 2019.
\bibitem[16]{Hibi0} Hibi, T.: {\em Some results on the Ehrhart polynomial of a convex polytope}, Discrete Math., {\bf 83}, p. 119-121, 1990.
\bibitem[17]{Hibi} Hibi, T.: {\em Dual polytopes of rational convex polytopes}, Combinatorica, {\bf 12} (2), p. 237-240, 1992.
\bibitem[18]{Hibi1} Hibi, T.: {\em A lower bound theorem for Ehrhart polynomials of convex polytopes}, Adv. Math., {\bf 105} , p. 162-165, 1994.
\bibitem[19]{K} Kouchnirenko, A.G.: {\em Poly\`edres de Newton et nombres de Milnor}, Invent. Math.,  {\bf 32}, p. 1-31, 1976.
\bibitem[20]{Mann} Mann, E.: {\em Cohomologie quantique orbifolde des espaces projectifs \`a poids}, arXiv:0510331.
\bibitem[21]{NS} Nemethi, A., Sabbah, C.: {\em  Semicontinuity of the spectrum at infinity}, Abh. math. Sem. Univ. Hamburg, {\bf 69}, p. 25-35, 1999.
\bibitem[22]{RV} Rodriguez-Villegas, F.: {\em On the zeros of certain polynomials}, Proceedings of the American Math. Soc. {\bf 130 (8)}, p. 2251-2254, 2002.
\bibitem[23]{Sab0} Sabbah, C.: {\em Monodromy at infinity and Fourier transform}, Pub. RIMS, Kyoto Univ., {\bf 33}, p. 643-685, 1998.
\bibitem[24]{Sab} Sabbah, C.: {\em Hypergeometric periods for a tame polynomial}, Portugalia Mathematicae, {\bf 63}, p. 173-226, 2006.
\bibitem[25]{Sab1} Sabbah, C.: {\em Some properties and applications of Brieskorn lattices}, J. Singul., {\bf 18}, p. 238-247, 2018.
\bibitem[26]{ScSt} Scherk, J., Steenbrink, J.H.M.: {\em On the mixed Hodge structure on the cohomology of the Milnor fibre}, Math. Ann., {\bf  271}, p. 641-665, 1985.
\bibitem[27]{Si} Simpson, C.: {\em The construction problem in K\"ahler geometry}, in {\em Different faces of geometry}, ed. S. K.
Donaldson, Y. Eliashberg and M. Gromov, p. 365-402, Int. Math. Ser. (N. Y.), 3, Kluwer/Plenum,
New York, 2004. 
\bibitem[28]{Sta} Stanley, R.: {\em Hilbert functions of graded algebras}, Adv. in Math., {\bf 28}, p. 57-83, 1978.
\bibitem[29]{Stan1} Stanley, R. P.: {\em On the hilbert function of a graded Cohen-Macaulay domain}, J. Pure Appl. Algebra, {\bf 73}, p. 307-314, 1991. 
\bibitem[30]{Stapledon} Stapledon, A.: {\em Weighted Ehrhart Theory and Orbifold Cohomology}, Adv. Math., {\bf 219},  p. 63-88, 2008. 
\bibitem[31]{Steen} Steenbrink, J.H.M.: {\em Mixed Hodge structure on the vanishing cohomology}, In: Holm, P.,ed. Real and complex singularities, 1976, p. 525-563. Oslo 1976. Alphen aan de Rijn Sijthoff-Noordhoff, 1977.
\bibitem[332]{Va} Varchenko, A. N.: {\em Asymptotic Hodge structure in the vanishing cohomology}, Isv. Akad. Nauk SSSR Ser. Mat., {\bf 18}, p. 469-512, 1982. 
\end{thebibliography}
\end{document}